\documentclass[11pt,reqno]{amsart}

\usepackage{amsmath,amsthm,amssymb,comment,fullpage}
\usepackage{braket}
\usepackage{mathtools}

\usepackage{caption}
\usepackage{times}
\usepackage[T1]{fontenc}
\usepackage{setspace}
\usepackage{mathrsfs}
\usepackage{latexsym}
\usepackage[dvips]{graphics}
\usepackage{epsfig}
\usepackage{hyperref, amsmath, amsthm, amsfonts, amscd, flafter,epsf}
\usepackage{amsmath,amsfonts,amsthm,amssymb,amscd}
\input amssym.def
\input amssym.tex
\usepackage{color}
\usepackage{hyperref}
\usepackage{url}
\usepackage{graphicx}
\usepackage[export]{adjustbox}
\newcommand{\bburl}[1]{\textcolor{blue}{\url{#1}}}

\newcommand{\burl}[1]{\textcolor{blue}{\url{#1}}}

\numberwithin{equation}{section}

\newtheorem{thm}{Theorem}[section]

\theoremstyle{plain}

\newtheorem{definition}[thm]{Definition}

\newtheorem{lemma}[thm]{Lemma}

\newtheorem{theorem}[thm]{Theorem}

\newtheorem{remark}[thm]{Remark}

\newcommand\be{\begin{equation}}
\newcommand\ee{\end{equation}}
\newcommand\bee{\begin{equation*}}
\newcommand\eee{\end{equation*}}
\newcommand\bea{\begin{eqnarray}}
\newcommand\eea{\end{eqnarray}}
\newcommand\beae{\begin{eqnarray*}}
\newcommand\eeae{\end{eqnarray*}}
\newcommand\bi{\begin{itemize}}
\newcommand\ei{\end{itemize}}
\newcommand\ben{\begin{enumerate}}
\newcommand\een{\end{enumerate}}
\newcommand\bc{\begin{center}}
\newcommand\ec{\end{center}}
\newcommand\ba{\begin{array}}
\newcommand\ea{\end{array}}

\newcommand\frakfamily{\usefont{U}{yfrak}{m}{n}}
\DeclareTextFontCommand{\textfrak}{\frakfamily}

\newcommand{\ncr}[2]{{#1 \choose #2}}

\newcommand{\hr}[1]{\href{#1}{\url{#1}}}

\title{Generalizing the German Tank Problem}

\author{Anthony Lee}
\email{\textcolor{blue}{\href{mailto:anthony\_lee24@milton.edu}{anthony\_lee24@milton.edu}}}
\address{Milton Academy, Milton, MA 02186}

\author[Miller]{Steven J. Miller}
\email{\textcolor{blue}{\href{mailto:sjm1@williams.edu}{sjm1@williams.edu}},  \textcolor{blue}{\href{Steven.Miller.MC.96@aya.yale.edu}{Steven.Miller.MC.96@aya.yale.edu}}}
\address{Department of Mathematics and Statistics, Williams College, Williamstown, MA 01267}


\date{\today}

\begin{document}

\maketitle

\begin{abstract}
The German Tank Problem dates back to World War II when the Allies used a statistical approach to estimate the number of enemy tanks produced or on the field from observed serial numbers after battles. Assuming that the tanks are labeled consecutively starting from 1, if we observe $k$ tanks from a total of $N$ tanks with the maximum observed tank being $m$, then the best estimate for $N$ is $m(1 + 1/k) - 1$. We explore many generalizations. First, we looked at the discrete and continuous one dimensional case. We attempted to improve the original formula by using different estimators such as the second largest and $L\textsuperscript{{\rm th}}$ largest tank, and applied motivation from portfolio theory by seeing if a weighted average of different estimators would produce less variance; however, the original formula, using the largest tank proved to be the best; the continuous case was similar. Then, we attempted to generalize the problem into two dimensions, where we pick pairs instead of points. We looked at the discrete and continuous square and circle variants. There were more complications, as we dealt with problems in geometry and number theory, such as dealing with curvature issues in the circle, and the problem that not every number is representable as a sum of two squares. In some cases, we concentrated on the large $N$ limit (with fixed $k$) by deriving approximate formulas. For the discrete and continuous square, we tested various statistics, but found that the largest observed component of our pairs is the best statistic to look at; the scaling factor for both cases is $(2k+1)/2k$. For the circle we used  motivation from the equation of a circle; for the continuous case, we looked at $\sqrt{X^2+Y^2}$ and for the discrete case, we looked at $X^2+Y^2$ and took a square root at the end to estimate for $r$. The discrete case was especially involved because we had to use approximation formulas that gave us the number of lattice points inside the circle. Interestingly, the scaling factors were different for the cases. Lastly, we generalized the problem into $L$ dimensional squares and circles. The discrete and continuous square proved similar to the two dimensional square problem. However, for the $L\textsuperscript{{\rm th}}$ dimensional circle, we had to use formulas for the volume of the $L$-ball, and had to approximate the number of lattice points inside it. The formulas for the discrete circle were particularly interesting, as there was no $L$ dependence in the formula.
\end{abstract}

\bigskip



\tableofcontents


\section{Introduction}\label{sec:introduction}

We study a problem where we are given some information about observations, and we have to estimate the number of objects. The motivation comes from the German Tank problem, a classic problem in probability. During World War II, Germans used powerful tanks to their advantage. To develop appropriate military strategies against the Germans, the Allies had to estimate the number of tanks produced on various battlefields. Initially, spies were used to figure out the number of tanks. However, using statistics to estimate the number of tanks proved to be much more accurate; see \cite{Ka} for a history of the problem. During battles, the Allies realized that the destroyed or captured tanks had serial numbers and that they could reverse engineer and use these to their advantage. Assuming that the serial numbers are consecutive\footnote{It is convenient to have the numbers in order; by looking at the serial number one can often quickly tell when it was made and thus when it may need certain types of repairs. However, as the previous work shows, this opens one up to disclosing more information than one would like, and in many applications now companies use formulas to determine the serial numbers, masking information.} and started from 1, the statisticians came up with a formula using the largest observed tank $m_k$ and the number of tanks observed $k$ to make a estimate for $N$, which we denote $\widehat{N}$: \bea \widehat{N} \ = \ m_k\bigg(1+\frac{1}{k}\bigg) \ - \ 1. \eea This formula proved to be very effective. A comparison between the actual and estimated production rate of tanks from \cite{Al} shows how accurate this formula was: while the intelligence estimated that Germans were producing 1,400 tanks per month, using the formula, the statisticians estimated that Germans were producing 256 tanks per month, and indeed, 255 tanks were made! The chart below compares the estimation from the statistical method and intelligence to the actual number of German Tanks \cite{Ru}.

\begin{figure}[h]
    \centering
    \includegraphics[width=10cm]{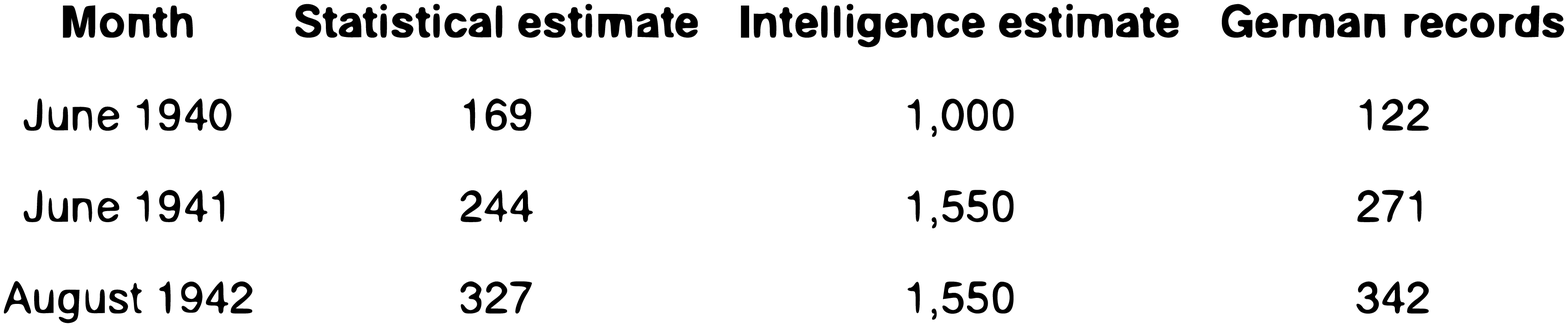}
    \caption{Statistics vs Intelligence.}
    \label{fig:Statistics vs Intelligence}
\end{figure}

We see that the statistical estimate is reasonably similar to the German records. However, the intelligence estimate is very off from the actual number of tanks. Had the Allies used the intelligence estimate, their strategy would have mislead them as they would have overly prepared defense to minimize damage from tanks, which could have left them vulnerable to other attacks. However, the statistical estimate provided them with reasonable estimates.

The German Tank problem is an excellent example of how statistical inference can be applied to real world problems. We attempt to generalize the well known German Tank Problem. Previously, Clark, Gonye, and Miller \cite{CGM} derived a more general formula where the smallest serial number isn't 1, but the tanks are still numbered consecutively. If the spread between the smallest and largest serial number is $s$, then their formula to estimate the number of tanks is
\bea \widehat{N} \ = \ s\bigg(1+\frac{2}{k-1}\bigg)\ - \ 1. \eea

We start by recalling the derivation of the original German tank problem, as we will be extending those calculations. We then see if we can improve the one dimensional formula by looking at different estimators and using motivation from portfolio theory in financial mathematics, specifically looking at weighted sums of estimators to see if we can find a new statistic that has the same mean of its predictions but smaller variance; see Appendix \ref{sec:appendixAportfoliotheory} for a review of portfolio theory and the construction of such statistics.   However, we find that the original German tank formula does better than formulas from these two approaches. We use code to check, and it is attached in Appendix \ref{appendixEmathematicacode}.

Generalizing further, we looked at what would happen if we modified some of the assumptions of the problem. First, we modified the condition that all serial numbers are consecutive integers and we looked at the standard German Tank Problem in the continuous setting. The problem changed slightly as we sample $k$ tanks from the range of $0$ to $N$, but the tanks can be real numbers. To find the answer of the question of finding the best estimator for $N$, we tried various statistics, starting with some obvious choices such as the largest observed tank, the second largest observed tank, and the weighted sum. Interestingly, the continuous case turned out to be very similar to the discrete case, as the best statistic to look at in the continuous case was the largest observed tank, as it produced a formula with the least variance. Furthermore, the scaling factor for the discrete and continuous case were both $(k+1)/k$, which shows the similarity between the two cases. The main conclusion here was that even in the continuous case, the largest observed tank is the best statistic to study.

Moving on from the one dimensional case, we generalized the German Tank problem further into two dimensions. In the two dimensional case, we select pairs instead of points, and the main condition is that we select $k$ pairs without replacement. Specifically, we looked at the square and circle, as we have explicit closed form and asymptotic expressions for the area and number of lattice points inside. The main strategy of deriving formulas was using the CDF method for the discrete and continuous case. (See Lemma \ref{2.7CDFdefinition} for details) Starting from the two dimensional case, we see geometry being involved in the calculations. For the discrete square, we looked at the square with bottom left vertex at $(1,1)$ and upper right vertex at $(N,N)$ and looked at the number of lattice points inside the square. From inspiration from the one dimensional problem, we looked at the largest observed component. Unfortunately, because the closed form expressions that we get are very complex, which makes it difficult to invert, we approximate by focusing on the main term and sometimes second order, as that’s enough to get a pretty good approximation of what’s going on. Thus, for the discrete square, we provide a approximate formula that still does a very good job estimating. The continuous square problem was much easier as calculating integrals was easier than calculating sums. We set the bottom left vertex as $(0,0)$ and the upper right vertex as $(N,N)$. We also looked at the largest observed component for this case too. However, instead of looking at the number of lattice points as we did in the discrete case, we looked at the area, which was much easier. Interestingly, the scaling factor for both the discrete and continuous square was the same, and we state both formulas.
\bea {\rm \textbf{Discrete Square}:} \ \widehat{N} \ = \ \frac{2k+1}{2k}(m-1) \ \ \ \ {\rm \textbf{Continuous Square}:} \ \widehat{N} \ = \ \frac{2k+1}{2k} \cdot m. \eea After looking at the square problem in the discrete and continuous case, we looked at the circle problem. For both the discrete and continuous cases, we set the circle so that it's center was at $(0,0)$ and that it had radius $r$. The statistics we looked at the for the discrete and continuous cases were different. For the continuous case, we looked at the largest observed value of $m_2=\sqrt{X^2+Y^2}$. The motivation comes from the standard form equation of the circle, which is $(x-x_1)^2 + (y-y_1)^2= r^2$ where $(x, y)$ are the arbitrary coordinates on the circumference of the circle, $r$ is the radius of the circle, and $(x_1, y_1)$ are the coordinates of the center of the circle. We used the CDF method with areas to calculate the formula.

Before we look at the discrete case, we see that we need to know the number of lattice points inside a circle with center $(0,0)$ and radius $r$. We visit the classic Gauss Circle problem and use the approximations and denote the error term using Big O notation to denote the number of lattice points inside the circle. To solve the discrete case, we looked at various statistics. First, we looked at the largest observed component, as we were able to obtain nice formulas in the square problem using the largest observed component. However, we faced problem of curvature because we had to split the ranges of the side of the square to see if the square completely fit in a circle or not. Thus, we decided that this isn't the best statistic to look at. We wanted to look at a statistic similar to $\sqrt{X^2+Y^2}$, the formula for the radius of the circle, but we had to make sure that the statistic that we were studying gave only discrete values as outcomes. Therefore, we chose to look at $m_1=X^2+Y^2$ as all values of $m$ are integers. At the end, we have to take the square root of $m$ because now, we are essentially estimating for $m^2$. Another complication that arises in this problem is that $X^2+Y^2\equiv 0,1,2$ (mod 4). We explain more about the number theory complication when we derive the formula. Also, because we used a asymptotic formula for the number of lattice points inside the circle, we weren't able to produce a exact formula, but we produce one that still gives accurate estimations. We see that the discrete and continuous formula is different, and we explain more details in Remark \ref{remark6.2}.
\bea {\rm \textbf{Discrete Circle}:} \ \widehat{r} \ = \  \sqrt{\frac{k+1}{k}(m_1-1)}. \ \ \ \ \ {\rm \textbf{Continuous Circle}:} \ \widehat{r} \ = \ \frac{2k+1}{2k} \cdot m_2. \eea

Finally, we generalized the problem into higher dimensions, $L$ dimensions. We looked at the $L^{\rm th}$ dimensional square and the $L^{\rm th}$ dimensional sphere, also known as the $L$-ball in the continuous and discrete setting. The problem changes slightly as we are selecting $k$ tuples of length $L$ without replacement. For the discrete and continuous $L^{\rm th}$ dimensional square cases, we looked at the largest observed component value from the tuples and derived the formulas. For the continuous $L^{\rm th}$ dimensional square, we obtained a exact formula, but for the discrete $L^{\rm th}$ dimensional square, we approximated the result using the main term. We see that the scaling factors for the discrete and continuous setting are the the same. \bea {\rm \textbf{Discrete L-dim Square}:} \  \widehat{N} \ = \ \frac{Lk+1}{Lk}(m-1). \ {\rm \textbf{Continuous L-dim Square}:} \ \widehat{N} \ = \ \frac{Lk+1}{Lk} \cdot m. \eea The difficulty is calculating the number of lattice points inside the $L$-ball. We use the known formula about the volume of the $L$-ball, and use this formula to estimate how many lattice points are contained inside using Big-O notation. Now, we look at various statistics. For our statistic in the discrete case, we looked at $m_1=X_1^2+X_2^2+\dots+X_L^2$, as this statistic guarantees that all $m$ values are discrete. Then after estimating, we took the square root of $m$ to find the estimate for $r$. For the continuous case, we looked at $m_2=\sqrt{X_1^2+X_2^2+\cdots+X_L^2}$ as we are allowed to get real numbers for $m$ values. The continuous $L$-ball problem was easier, as we could plug in the equation for the $L$-ball to get the volume instead of having to estimate the number of lattice points. We state the formulas we derived: \bea {\rm \textbf{Discrete L-ball}:} \  \widehat{r} \ = \ \sqrt{\frac{k+1}{k} \cdot (m_1-1)}. \ {\rm \textbf{Continuous L-ball}:} \ \widehat{r} \ = \ \frac{Lk+1}{Lk}\cdot m_2. \eea

\section{Preliminaries}\label{sec:preliminaries}

We quickly review some needed results from probability, combinatorics, and integration.

\subsection{Probability Review}

We list a few standard results from probability; for proofs see for example \cite{Mi, Sh}.

\begin{definition}\label{2.1variancedefinition} The variance for a random variable $X$ is the average of the squared difference from the mean, $\mathbb{E}[X]$: \bea {\rm Var}(X) \ := \ \mathbb{E}[(X-\mathbb{E}[X])^2]. \eea
\end{definition}

\begin{lemma}\label{2.2variancelemma} The variance can be computed by
\bea {\rm Var}(X) \ = \ \mathbb{E}[X^2] \ - \ \mathbb{E}[X]^2. \eea
\end{lemma}

See Lemma \ref{lem:varformula} in Appendix \ref{sec:appendixBprob} for a proof.

\begin{definition}\label{2.3covariancedefinition} For two jointly distributed real valued random variables $X$ and $Y$, the covariance is defined as the expected value of the product of their deviations from their individual expected values: \bea {\rm Cov}(X,Y) \ := \ \mathbb{E}\left[ \left(X-\mathbb{E}[X]\right) \cdot \left(Y-\mathbb{E}[Y]\right)\right]. \eea
\end{definition}

\begin{lemma}\label{2.4covariancelemma} The covariance can be computed by
\bea {\rm Cov}(X,Y)  \ = \ \mathbb{E}[XY] - \mathbb{E}[X]\mathbb{E}[Y]. \eea
\end{lemma}

See Lemma \ref{lem:covariance} in Appendix \ref{sec:appendixBprob} for a proof.

\begin{theorem}\label{2.5linearityofexpectation}(Linearity of expectation)
Let $X_1, \dots, X_n$ be random variables, and let $g_1, \dots, g_n$ be functions such that $\mathbb{E}[|g_i(X_i)|]$ exists and is finite, and let $a_1, \dots, a_n$ be any real numbers; note the random variables do not have to be independent. Then \bea \mathbb{E}\left[a_1g_1(X_1)+ \dots + a_ng_n(X_n)\right] \ = \ a_1\mathbb{E}[g_1(X_1)] + \cdots + a_n\mathbb{E}[g_n(X_n)]. \eea
\end{theorem}

\begin{theorem}\label{2.6jointprobabilitydensityfunction}(Joint Probability Density Function)
Let $X_1, X_2, \dots,X_n$ be continuous random variables with densities $f_{X_1}, f_{X_2}, \dots, f_{X_n}$ defined on $\mathbb{R}$. The joint density function of the tuple $(X_1,\dots,X_n)$ is a non-negative integrable function $f_{X_1,X_2,\dots,X_n}$ such that for every nice set $S \subset \mathbb{R}^n$ we have \bea {\rm Prob}((X_1,\dots,X_n) \in S) \ = \ \idotsint_S f_{X_1,X_2,\dots,X_n}(x_1,\dots,x_n) \,dx_1 \cdots dx_n, \eea and \bea f_{X_i}(x_i) \ = \ \int_{x_1, \dots, x_{i-1}, x_{i+1}, \dots, x_n=-\infty}^\infty f_{X_1,X_2,\dots,X_n}(x_1,\dots,x_n) dx_1 \cdots dx_{i-1} dx_{i+1} \cdots dx_n. \eea For discrete random variables, we can just replace integrals with sums. \end{theorem}

\begin{definition}\label{2.7CDFdefinition}
The cumulative distribution function (CDF) of a random variable $X$ with density $f$, denoted $F$, is given by \bea F(x) \ := \ {\rm Prob}(X \le x) \ = \ \int_{-\infty}^{x} f(t) \,dt, \ {\rm for \ any} \ x \in \mathbb{R}. \eea
\end{definition}

In many situations it is easy to compute the CDF, and thus by using the Fundamental Theorem of Calculus we can determine the probability density function. In the continuous case it is the derivative of the CDF, in the discrete case when the values of our random variable are non-negative integers, then the probability of $m$ is $F(m) - F(m-1)$.

\subsection{Analysis Review}

Unfortunately many of the sums we encounter are very complex and difficult to invert when expressed in a closed form, and thus we estimate using the main term and the second order. The following frequently provides a good, easily computed bound; see for example \cite{YK}.

\begin{theorem}\label{2.8EulerMaclaurin}(Euler-Maclaurin formula) For $p$ a positive integer and a function $f(x)$ that is $p$ times continuously differentiable on the interval $[a,b]$, we have \bea \sum_{i= a}^{b} f(i) & \ = \ & \int_{a}^{b} f(x) \,dx + \frac{f(a)+f(b)}{2} \ + \ \sum_{q=1}^{\lfloor \frac{p}{2} \rfloor}\frac{B_{2q}}{(2q)!}(f^{2q-1}(b)-f^{2q-1}(a))+R_p, \eea and \bea |R_p| \ \le \ \frac{2\zeta(p)}{(2\pi)^p} \int_{m}^{n}|f^{(p)}(x)| \,dx. \eea
\end{theorem}

The estimation below is needed in computing upper and lower bounds in applications of the Euler-Maclaurin formula later.

\begin{lemma}\label{2.9boundslemma} For $m,L \geq 0$ and $k \geq 1$,
\bea \bigg(m^{Lk} - m^{Lk-L}\bigg(\frac{k(k-1)}{2}\bigg)\bigg)\ \le \  m^L(m^L-1)\cdots(m^L-(k-1) \le m^{Lk}. \eea
\end{lemma}

\begin{proof}
The upper bound follow trivially as  \bea m^L(m^L-1)\cdots(m^L-(k-1)) \ \le \ (m^L)(m^L)\cdots (m^L). \eea To justify the lower bound, we want to prove the inequality \bea m^L(m^L-1) \cdots (m^L-(r-1)) \ \geq \ m^{Lr}-\frac{r(r-1)}{2}m^{Lr-L}. \eea The base case is satisfied when $r=1$ as $m^2 \ \geq \ m^2$. We assume that the case when $r=k$ is true: \bea m^L(m^L-1) \cdots (m^L-(k-1)) \ \geq \ m^{Lk}-\frac{k(k-1)}{2}m^{Lk-L}. \eea We must show it holds when $r=k+1$; namely we must show \bea m^L(m^L-1) \cdots (m^L-k) \ \geq \ m^{Lk+L} \ - \ \frac{(k+1)k}{2}m^{Lk}. \eea This follows immediately by substitution and expansion: \bea m^L(m^L-1) \cdots (m^L-(k-1))(m^L-k) & \ \geq \ & \bigg(m^{Lk}-\frac{k(k-1)}{2}m^{Lk-L}\bigg)(m^L-k) \nonumber \\ & \geq & m^{Lk+L} \ - \ \frac{k(k+1)}{2} m^{Lk} \ + \ \frac{k^2(k-2)}{2}m^{Lk-L} \nonumber \\ & \geq & m^{Lk+L} \ - \ \frac{k(k+1)}{2}m^{Lk}. \eea
\end{proof}

Finally, in many of our generalizations it is impossible to obtain simple closed form expressions such as those in the original problem. We thus investigate the large $N$ limit, and big-Oh notation is useful in isolating the main term from lower order terms which have negligible effect.

\begin{definition}\label{2.10bigOnotation}(Big-Oh Notation)
Suppose $f(x)$ and $g(x)$ are two functions defined on the real numbers. We write $f(x) =  O(g(x))$ (read ``$f$ is Big-Oh of $g$'') if there exists a positive constant $C$ such that $|f(x)| \le C g(x)$ is satisfied for all sufficiently large $x$.
\end{definition}

\subsection{Combinatorial Results}

We list four useful identities that are needed in later calculations. These are generalizations of the famous hockey stick identity from Pascal's triangle; that identity is responsible for the closed form expression in the original problem, and these identities play a similar role in our generalizations. The proofs for these identities can be found in the Appendix \ref{sec:appendixCproofsofidentities}.

\ \\

\noindent \textbf{Identity I}: For all $N \geq k$, \bea \sum_{m=k}^{N} \binom{m-b}{k-c} & \ = \ & \binom{N-b+1}{k-c+1} \ - \ \binom{k-b}{k-c+1}. \eea

\ \\

\noindent \textbf{Identity II}: For all $N \geq k$, \bea \sum_{m=k-a+1}^{N-a+1} m \frac{\binom{m-1}{k-a}\binom{N-m}{a-1}}{\binom{N}{k}} & \ = \ & \frac{(N+1)(k-a+1)}{k+1}.\eea

\ \\

\noindent \textbf{Identity III}: For all $N \geq k$, \bea \sum_{m=k-a+1}^{N-a+1} m^2 \frac{\binom{m-1}{k-a}\binom{N-m}{a-1}}{\binom{N}{k}} & \ = \ & \frac{(k-a+1)(k-a+2)(N+2)(N+1)}{(k+2)(k+1)} \nonumber \\ && \ \ \ \ - \ \frac{(N+1)(k-a+1)}{k+1}. \eea

\ \\
\noindent \textbf{Identity IV}: For all $k \geq 0$, \bea \sum_{i=0}^{k} \binom{a+i}{a} \binom{b+k-i}{b} \ = \ \binom{a+b+k+1}{a+b+1}. \eea

\section{Derivation of Original German Tank Problem}

We derive below the formula for the original German tank problem where $m$ represents the largest tank observed, $k$ the number of tanks observed, and $N$ the number of tanks (which are numbered consecutively from 1 to $N$). We use $\widehat{N}$ to represent our estimate for the total number of tanks, $N$. It is useful to run through this argument before we generalize the inference problem, and we follow the exposition in \cite{CGM}. We prove  \bea \widehat{N} \ = \ m_k\bigg(1+\frac{1}{k}\bigg) \ - \ 1. \eea

We quickly check some extreme cases to see if this formula is reasonable. When we observe just one tank, the estimation formula is $2m - 1$. We expect the observation to be $N/2$, so multiplying by 2 is reasonable. When we observe all $N$ tanks (so $k = N$), the estimation is just $m$, which is also reasonable because we know the number all the tanks. We use Pascal's identity and the hockey stick identity frequently in calculations; see Appendix \ref{sec:appendixBprob} for proofs.

\subsubsection{The PDF of the Largest Observed Tank}

Let $M_k$ be the random variable for the largest tank observed, and let $m_k$ be its observed value. Thus $m_k$ is the largest tank serial number that we observe. Our goal is to find a formula to estimate $N$ from $m_k$ and $k$. We first compute the PDF for $M_k$.

\begin{lemma}
For $k \le m_k \le n$, \bea {\rm Prob}(M=m_k) \ = \ \frac{\binom{m_k-1}{k-1}}{\binom{N}{k}}. \eea
\end{lemma}

\begin{proof}
We know that the total number of ways to select $k$ tanks from $N$ possible tanks is $\binom{N}{k}$. If the largest tank we observe is $m$, then we have to choose $k-1$ tanks from $m_k-1$ possibilities. Therefore, the probability that the largest tank we observe is $m$ is $\binom{m_k-1}{k-1}/ \binom{N}{k}$, thus proving the claim.
\end{proof}

\begin{remark} It is worth remarking that the process of deriving the formula is different from the application. In the application, we use tanks we observe, $m_1,m_2, \dots, m_k$ and estimate for $N$ by applying the formula. However, when we are deriving the formula, we use $N$ and $k$ to compute the expected value of $M_k$, and then we invert the equation to find the formula for $N$.
\end{remark}

\subsubsection{Derivation}

Now we calculate the expected value of $M_K$ in order to find a equation with $N$ and $k$. To calculate the expected value, we multiply each value of the random variable by its probability and add the products. We expand the binomials, and regroup so that we can use the hockey stick identity. We find \bea \mathbb{E}[M_k] & \ = \ & \sum_{m_k=k}^{N} m_k {\rm Prob}(M_k=m_k) \nonumber\\ &=& \sum_{m_k=k}^{N} m_k \frac{\binom{m_k-1}{k-1}}{\binom{N}{k}} \nonumber\\
&=& \frac{1}{\binom{N}{k}} \sum_{m_k=k}^{N} \frac{m_k!}{(k-1)!(m_k-k)!} \frac{k}{k} \nonumber\\
&=& \frac{k}{\binom{N}{k}} \sum_{m_k=k}^{N} \frac{m_k!}{(k)!(m_k-k)!} \nonumber\\
&=& \frac{k}{\binom{N}{k}} \sum_{m_k=k}^{N} \binom{m_k}{k} \nonumber\\
&& {\rm We \ use \ the \ hockey \ stick \ identity, \ recall \ Lemma \ \ref{lem:fourhockeystickidentity}, \ to \ evaluate \ the \ sum} \nonumber\\
&=& \frac{k}{\binom{N}{k}} \binom{N+1}{k+1} \nonumber\\
&=& k \Bigg(\frac{N+1}{k+1}\Bigg). \eea
Using our formula for $M_k$ as a function of $N$ and $k$, we can invert and obtain the formula for $N$ in terms of $M_k$ and $k$:
\bea \widehat{N} \ = \ \mathbb{E}[M_k] \Bigg(\frac{k+1}{k}\Bigg) \ - \ 1. \eea
To obtain our estimate $\widehat{N}$ we substitute the observed $m_k$ for $\mathbb{E}[M_k]$:
\bea \widehat{N} \ = \ m_k \Bigg(\frac{k+1}{k}\Bigg) \ - \ 1. \eea
To see how well this formula does, we paste results of code. We see the mean and variance when we plug in values of $N$ and $k$.
\begin{figure}[h]
    \centering
    \includegraphics[width=16cm, frame]{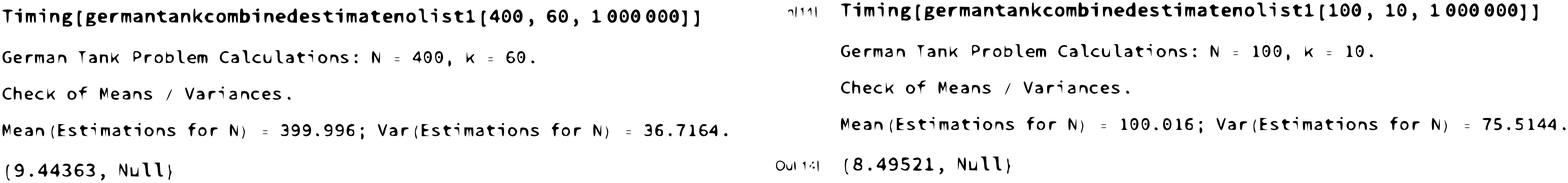}
    \caption{Results for formula using largest tank.}
    \label{fig:One dimensional Problem code}
\end{figure}
\section{Estimating with more tanks}

There are two natural approaches to try to improve the our prediction for the original problem. One is to find another estimator; perhaps the mean or median might do better. We pursue this and investigate statistics derived from the $L\textsuperscript{{\rm th}}$ largest tank, value does better.

Another approach is to use some motivation from portfolio theory; see for example \cite{H}. Imagine we have two independent stocks with the same expected return. By looking at a linear combination, the combined portfolio will still have the same expected return, but if the weights are chosen properly a smaller variance; we provide details in Appendix \ref{sec:appendixAportfoliotheory}. The idea of weighing two random variables to create a combined one with less variance is a common method, and we apply that to our problem to see if we can improve the quality of the estimator for the German tank problem by looking at a combination of two tanks. The computation is a bit more involved than the two stock example, as the values of the two tanks are not independent; if we know the largest tank value is $m_k$ then clearly $m_{k-1} < m_k$. We provide complete details of the analysis of exploring statistics involving the two largest observed serial numbers, and find that the quality of the estimator is \emph{not} improved by incorporating both values.

\subsection{Estimation from Various Tanks}

Let $M_{k-1}$ be the random variable for the value of the second largest tank observed and let $m_{k-1}$ be the value we observe. We find the probability that the second largest tank is $m_{k-1}.$ We claim that for $k-1 \le m_{k-1} \le n-1$, the probability that $M_{k-1} =  m_{k-1}$ is \bea {\rm Prob}(M_{k-1}=m_{k-1}) \ = \ \frac{\binom{m_{k-1}-1}{k-2}\binom{N-m_{k-1}}{1}}{\binom{N}{k}}. \eea Clearly the probability is zero for tanks outside this range. There are $\binom{N}{k}$ ways to choose $k$ tanks from $N$. We need $k-2$ tanks to be smaller than $m_{k-1}-1$; there are $\binom{m_{k-1}-1}{k-2}$ ways for that to happen. We then have to choose tank $m_k$, which has to be larger than tank $m_{k-1}$. Thus, the range from $m_k$ goes from $m_{k-1}+1$ to $N$, and there are $\binom{N-m_{k-1}}{1}$ ways to do this. \hfill $\Box$

We can now calculate the expected value of $M_{k-1}$: \bea \mathbb{E}[M_{k-1}] & \ = \ & \sum_{m_{k-1}=k-1}^{N-1} m_{k-1} {\rm Prob}(M_{k-1} \ = \ m_{k-1}) \nonumber \\ &=& \sum_{m_{k-1}=k-1}^{N-1} m_{k-1} \frac{\binom{m_{k-1}-1}{k-2} (N-m_{k-1})}{\binom{N}{k}} \nonumber \\ &=& \frac{1}{\binom{N}{k}}\sum_{m_{k-1}=k-1}^{N-1} \frac{k-1}{k-1} \frac{m_{k-1}!}{(k-2)!(m_{k-1}-k+1)!}(N-m_{k-1}) \nonumber \\ &=& \frac{k-1}{\binom{N}{k}}\sum_{m_{k-1}=k-1}^{N-1}  \frac{m_{k-1}!}{(k-1)!(m_{k-1}-k+1)!}(N-m_{k-1}) \nonumber \\ &=& \frac{k-1}{\binom{N}{k}}\sum_{m_{k-1}=k-1}^{N-1} \binom{m_{k-1}}{k-1}\binom{N-m_{k-1}}{1} \nonumber \\ && {\rm We \ use \ Identity \ IV} \nonumber \\ &=& \frac{k-1}{\binom{N}{k}} \binom{N+1}{k+1} \nonumber \\ &=& (N+1)\frac{k-1}{k+1}. \eea Therefore, the formula for our estimator of $N$, in terms of $m$ and $k$, coming from $M_{k-1}$ is \bea \widehat{N} \ = \ m_{k-1} \frac{k+1}{k-1} \ - \ 1. \eea
We attach some code to see the mean and variances of the formula using the second largest tank.
\begin{figure}[h]
    \centering
    \includegraphics[width=16cm, frame]{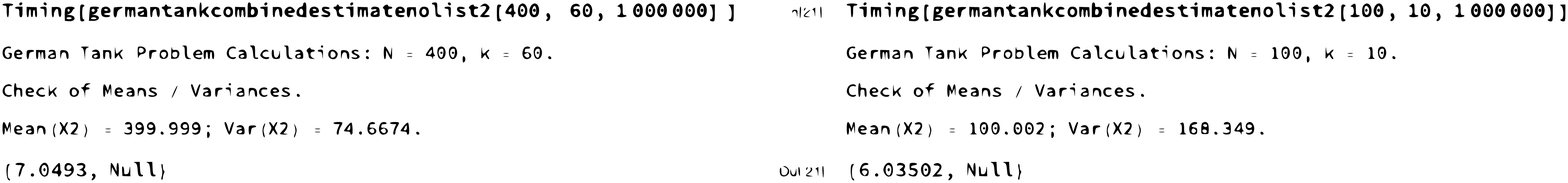}
    \caption{Results for formula using second largest tank.}
    \label{fig:One dimensional Problem code}
\end{figure}
\begin{remark}
Comparing this result to the variance of the formula using the largest tank, we see that the formula using the second largest tank has higher variance. Later, we confirm this finding by theory, and not just by code.
\end{remark}
Calculating similarly, we obtain the formula for the $L\textsuperscript{th}$ largest tank (see Appendix \ref{sec:appendixDlargesttanks} for details): \bea \widehat{N} \ = \ m_{k-L+1}\frac{k+1}{k-L+1} \ - \ 1. \eea

We calculate the variances for statistics arising from the different tank (from largest observed to smallest); we scale each of these so that the mean is $N$, and see which has the least variance. We let $X_k$ be the estimation from using the largest tank, $X_{k-1}$ from using the second largest tank, and so on. To calculate the variances of these formulas, we calculate the variances of $M_k$, the largest tank, $M_{k-1}$, the second largest tank, and so on; we then scale by a multiplicative factor to get the variances we want.

\subsubsection{Variance of $M_k$}: We use Lemma \ref{2.2variancelemma}, to compute the variance. We find
\bea {\rm Var}(M_k) & \ = \ & \mathbb{E}[M_k^2] \ - \ \mathbb{E}[M_k]^2 \nonumber \\ &=& \sum_{m_{k}=k}^{N}m_k^2 {\rm Prob(M_k \ = \ m_k)} \ - \ \bigg[\sum_{m_{k}=k}^{N}m_k {\rm Prob(M_k \ = \ m_k)}\bigg]^2 \nonumber \\ && {\rm We \ know \ the \ second \ term, \ as \ we \ calculated \ it \ earlier.} \nonumber \\ && {\rm We \ just \ have \ to \ calculate \ the \ first \ term.} \nonumber \\ &=& \sum_{m_{k}=k}^{N}m_k^2 \frac{\binom{m_k-1}{k-1}}{\binom{N}{k}} \ - \ \bigg[\frac{k(N+1)}{k+1}\bigg]^2 \nonumber \\ && {\rm We \ use \ Identity \ III \ to \ calculate \ the \ first \ term.} \nonumber \\ && {\rm After \ substituting \ in \ and \ simplifying, \ we \ get} \nonumber \\ &=& = \ N^2 \ - \ (2N-1)\frac{N-k}{k+1} \ \nonumber \\ && \ \ \ \ \ + \ 2\frac{(N-k)(N-k-1)}{(k+1)(k+2)} \ - \ \frac{k^2(N+1)^2}{(k+1)^2} \nonumber \\ &=&  \frac{(k)(N-k)(N+1)}{(k+1)^2(k+2)}. \eea
Now that we have the variance for $M_k$, we easily scale the formula and obtain the variance of $X_k  = m_k (k+1)/k$, as the variance of $a$ times $X$ is $a^2$ times the variance of $X$: \bea {\rm Var} (X_k) & \ = \ & {\rm Var} (M_k) \cdot \frac{(k+1)^2}{k^2} \nonumber \\ &=& \frac{(k)(N-k)(N+1)}{(k+1)^2(k+2)} \cdot  \frac{(k+1)^2}{k^2} \nonumber \\ &=& \frac{(N-k)(N+1)}{(k)(k+2)}. \eea Through similar calculations, we obtain the variance of $X_{k-1}$ (see Appendix \ref{sec:appendixDlargesttanks} for details): \bea {\rm Var} (X_{k-1}) & \ = \ & \frac{2(N-k)(N+1)}{(k+2)(k-1)}. \eea

We see that the estimator using $m_{k-1}$, and more generally using $m_{k-L+1}$, is worse than using $m_k$, as the variance is larger. We can simply plug in some values of $N$ and $k$ and compare the variances. Also, we see that the variance for $X_{k-1}$ is roughly two times the variance of $X_k$, which shows that $X_k$ is a better statistic. We can also compare variances easily by writing some code and numerically exploring. We made a simulation where we sample the $k$ tanks from 1 to $N$ when all tanks are equally likely to be seen. This simulation is attached in the Appendix \ref{appendixEmathematicacode}. Thus, we conclude that if we are only going to use one observed value, it is best to use the largest.

This leads to our second question: can we do better if we create a statistic combining two or more observed values?

\subsection{Weighted Statistic}

Previously, we've only considered using one tank to estimate for $N$. In this section, we create a statistic using the largest and second largest tank values, and show that this weighted statistic isn't better.

Before we state the weighted formula, we set our notation.
\begin{itemize}
  \item $N$ \ = \ Total number of tanks
  \item $k$ \ = \ Numbers of tanks observed
  \item $M_k$ \ = \ Largest Observed Tank
  \item $M_{k-1}$ \ = \ Second largest Observed Tank
  \item $X_k$ \ = \ Statistic to estimate $N$ using $M_k$: $M_k(\frac{k+1}{k})$ \ - \ 1
  \item $X_{k-1}$ \ = \ Statistics to estimate $N$ using $M_{k-1}$: $M_{k-1}(\frac{k+1}{k-1})$ \ - \ 1
\end{itemize}

We state the weighted statistic. Let $\alpha \in [0,1]$ and define the weighted statistic $X_\alpha$ by \bea
X_\alpha \ := \ \alpha X_k \ + \ (1-\alpha) X_{k-1}. \eea From the formula, we see that when $\alpha$ is 1, this collapses to the formula for $N$ using $m_k$, and similarly when $\alpha$ is 0, we get the $m_{k-1}$ formula. In order to see if there is a better estimation, we find the optimal $\alpha$ value, the value that minimizes the variance of $X_\alpha$. If the optimal $\alpha$ value is 1, then we can conclude that the $m_k$ formula is the best we can do to estimate for $N$. We compute the variance of $X$: \bea {\rm Var \ X}_\alpha \ = \ \alpha^2 {\rm Var \ X_k} \ + \ (1-\alpha)^2 {\rm Var \ X_{k-1}} \ + \ 2 \ \alpha (1-\alpha){\rm Cov}(X_k,X_{k-1}). \eea As We have already calculated the variances of $X_k$ and $X_{k-1}$, we just have to calculate the covariance term.

\subsubsection{\textbf{Covariance term}}

We calculate the term, ${\rm Cov}[X_k,X_{k-1}]$ separately first. By Lemma \ref{2.4covariancelemma}, we have \bea {\rm Cov}[X_k,X_{k-1}] & \ = \ & \mathbb{E}[X_k \cdot X_{k-1}] \ - \ \mathbb{E}[X_k] \cdot \mathbb{E}[X_{k-1}]. \eea Recall that \bea X_k \ = \ m_k\bigg(\frac{k+1}{k}\bigg)-1 \ , \ \ \ \  X_{k-1} \ = \ m_{k-1}\bigg(\frac{k+1}{k-1}\bigg)-1. \eea We know that the second term of the covariance, $\mathbb{E}[X_k] \cdot \mathbb{E}[X_{k-1}]$ is $N^2$ because they are both estimation formulas. Thus, we only have to calculate the first term.

We use linearity of expectation to expand. Recall Theorem \ref{2.5linearityofexpectation}, linearity of expectation. \bea \mathbb{E}[X_k \cdot X_{k-1}] & \ = \ & \mathbb{E}\bigg[\bigg(m_k\bigg(\frac{k+1}{k}\bigg)-1\bigg) \cdot \bigg(m_{k-1}\bigg(\frac{k+1}{k-1}\bigg)-1\bigg)\bigg] \nonumber \\ &=& \frac{(k+1)^2}{k(k-1)}\mathbb{E}[m_k \cdot m_{k-1}] \ - \ \frac{k+1}{k}\mathbb{E}[m_k] \ - \ \frac{k+1}{k-1}\mathbb{E}[m_{k-1}] \ + \ \mathbb{E}[1]. \eea

We calculate term by term. We use the joint PDF, recall Theorem \ref{2.6jointprobabilitydensityfunction}, to calculate $\mathbb{E}[m_k\cdot m_{k-1}]$. \bea \mathbb{E}[M_k \cdot M_{k-1}] & \ = \ & \sum_{m_{k}=k}^{N}\sum_{m_{k-1}=k-1}^{m_k-1}m_{k}m_{k-1}  {\rm Prob}(M_{k-1}=m_{k-1},M_{k}=m_{k}) \nonumber \\ && {\rm This \ is \ the \ joint \ probability \ density \ function \ of \ m_k \ and \ m_{k-1}} \nonumber \\ &=& \sum_{m_{k}=k}^{N}\sum_{m_{k-1}=k-1}^{m_k-1}m_{k}m_{k-1}  \frac{\binom{m_{k-1}-1}{k-2}}{\binom{N}{k}} \nonumber \\ && {\rm The \ probability \ is \ the \ following \ because \ after \ selecting \ m_k \ and \ m_{k-1}, \ we} \nonumber \\ && \ \ \ \ {\rm have \ to \ select \ (k-2) \ more \ tanks \ from \ the \ possible \ (m_{k-1}-1) \ tanks.} \nonumber \\ && \ \ \ \ \ {\rm The \ range \ of \ m_{k-1} \ is \ dependent \ on \ the \ value \ of \ m_k \ because} \nonumber \\ && \ \ \ \ \ {\rm m_{k-1} \ \le \ m_k \ is \ always \ satisfied.} \nonumber \\ &=& \frac{1}{\binom{N}{k}}\sum_{m_{k}=k}^{N}m_k\sum_{m_{k-1}=k-1}^{m_k-1}m_{k-1}  \binom{m_{k-1}-1}{k-2} \nonumber \\ &=& \frac{k-1}{\binom{N}{k}}\sum_{m_{k}=k}^{N}m_k\sum_{m_{k-1}=k-1}^{m_k-1}\binom{m_{k-1}}{k-1} \nonumber \\ &=& \frac{k-1}{\binom{N}{k}}\sum_{m_{k}=k}^{N}m_k\binom{m_k}{k} \nonumber \\ && {\rm We \ use \ Identity \ II \ to \ calculate} \nonumber \\ &=& \frac{(k-1)(N+2)(N+1)}{(k+2)}  \ - \ \frac{(k-1)(N+1)}{k+1}. \eea

Now that we've calculated the joint PDF, we know all the terms and find
\bea \mathbb{E}[X_k \cdot X_{k-1}] & \ = \ & \frac{(k+1)^2}{k(k-1)}\bigg[\frac{(k-1)(N+2)(N+1)}{(k+2)} - \frac{(k-1)(N+1)}{k+1}\bigg] \nonumber \\ && \ \ \ \ \ - \ \frac{k+1}{k} \frac{(N+1)k}{k+1} \ - \ \frac{k+1}{k-1}\frac{(N+1)(k-1)}{k+1} \ + \ 1, \eea and of course from our normalizations we have
\bea \mathbb{E}[X_k] \cdot \mathbb{E}[X_{k-1}] & \ = \ & N^2. \eea
\bea \mathbb{E}[X_k \cdot X_{k-1}] - \mathbb{E}[X_k] \cdot \mathbb{E}[X_{k-1}] & \ = \ & \frac{(k+1)^2}{k}\frac{(N+2)(N+1)}{(k+2)}  \ - \ \frac{(k+1)}{k}{(N+1)} \ - (N+1)^2 \nonumber \\ &=& \frac{(n+1)(n-k)}{k(k+2)}. \eea

\subsubsection{Finding optimal alpha}

Now that we've calculated ${\rm Var}(X_k)$, ${\rm Var}(X_{k-1})$, and ${\rm Cov}(X_k,X_{k-1})$, we find the optimal $\alpha$ value (that minimizes the variance of $X_\alpha$). To find the optimal alpha value, we optimize by taking the derivative and finding the unique value that gives the least variance. Let $\alpha_{k,{k-1}}$ denote the specific value of $\alpha$ that minimizes the variance. We have
\bea {\rm Var}(X_\alpha) & \ = \ & \alpha^2 {\rm Var}(X_k) \ + \ (1-\alpha)^2 {\rm Var}(X_{k-1}) + 2\alpha(1-\alpha){\rm Cov}(X_k,X_{k-1}) \nonumber \\ &=& \alpha^2 ({\rm Var}(X_k)+{\rm Var}(X_{k-1})-2 \  {\rm Cov}(X_k,X_{k-1})) \nonumber \\ && \ \ \ \ \ \ \ \ \ + \ 2\alpha ({\rm Cov}(X_k,X_{k-1}) \ - \ {\rm Var}(X_{k-1})) \ + \ {\rm Var}(X_{k-1}). \eea Taking the derivative with respect to $\alpha$ yields \bea {\rm Var}(X_\alpha)' & \ = \ & 2\alpha_{k,{k-1}}({\rm Var}(X_k)+{\rm Var}(X_{k-1})-2 \  {\rm Cov}(X_k,X_{k-1}) \nonumber \\ && \ \ \ \ \ \ + \ 2 \  ({\rm Cov}(X_k,X_{k-1}) \ - \ {\rm Var}(X_{k-1})). \eea

Because we want to find the optimal $\alpha$ value, we solve for ${\rm Var}(X)'  = 0$ (and of course also check the endpoints of $\alpha = 0$ or $1$). After substituting in the formulas and doing some algebra, we see that the optimal alpha value is 1, which is in fact one of our endpoints and corresponds to only using the largest tank. Thus if we use both the largest and second largest values observed we do worse than just using the largest:
\bea \alpha_{k,{k-1}} \ = \ \frac{{\rm Var}(X_{k-1}) \ - \ {\rm Cov}(X_k,X_{k-1})}{{\rm Var}(X_k)+{\rm Var}(X_{k-1})-2 \  {\rm Cov}(X_k,X_{k-1})} \ = \ 1. \eea

Similar calculations hold for other weighted combinations. Therefore, in the discrete one dimensional German Tank problem where we sample without replacement, we conclude that the formula using the largest tank does the best job.

\section{Continuous One Dimensional Problem}

We now explore our first generalization of the German Tank Problem and consider a continuous one dimensional analogue. In the original formulation, the serial numbers of the tanks were all integers drawn from 1 to $N$. We now consider a continuous version, where we select $k$ tanks from the interval $[0, N]$ with $N$ unknown. Thus our goal is to find a statistic to estimate $N$. We discuss the effectiveness of various statistics and compare the continuous formulas to the discrete ones; the scaling factor is the same in the continuous and discrete one dimensional problems.

\subsection{Formulas Using Largest and Second Largest Observations.}

We begin by estimating using the largest and second largest tanks by using the CDF method to find the PDF. In the continuous case, to find the PDF, we take the derivative of the CDF.

\subsubsection{Formula from Largest}: We first find the CDF by computing the probability that all the tanks are at most $m_k$: \bea {\rm Prob}(M_k \le m_k) \ = \ \bigg(\frac{m_k}{N}\bigg)^k; \eea this is because in the continuous case we can view all $k$ observations as independent, and the probability any is at most $m_k$ is just $m_k/N$.

Taking the derivative gives the PDF: \bea f(m_k) \ = \ PDF_{M_k}(m_k) \ = \ CDF_{M_k}(m_k)' \ = \ \frac{k m_k^{k-1}}{N^k}. \eea

Now that we have the PDF, we calculate the expected value of $m_k$: \bea \mathbb{E}[M_k] & \ = \ & \int_{0}^{N} m_k f(m_k) \,d m_k \nonumber \\ &=& \int_{0}^{N} m_k \frac{k m_k^{k-1}}{N^k} \,d m_k \nonumber \\ &=& \frac{k}{N^k} \int_{0}^{N} m_k^k \,d m_k \nonumber \\ &=& \frac{k}{k+1}\cdot N. \eea Thus we obtain
\bea \widehat{N} \ = \ m \cdot \bigg(1+\frac{1}{k}\bigg). \eea

\begin{remark}
Comparing the continuous formula to the discrete, the only difference is that there is a $-1$ in the discrete formula. This difference is due to in the discrete case we sample from tanks numbered from 1 to $N$ while in the continuous case wehave chosen to have the interval to be $[0, N]$.
\end{remark}

\subsubsection{Formula using Second Largest}: We first find the CDF by computing the probability that the second largest tank is at most $m_{k-1}$. There are two possibilities -- all the tanks are less than $m_{k-1}$, or one tank (and there is $\binom{N}{1}$ ways to choose which of the $k$ tanks that is) is larger than $m_{k-1}$ and the rest are $m_{k-1}$ or smaller. Thus \bea CDF_{M}(m_{k-1}) & \ = \ & {\rm Prob}(M_{k-1} \le X) \nonumber \\ &=& \bigg(\frac{X}{N}\bigg)^k \ + \ \binom{k}{1} \bigg(\frac{X}{N}\bigg)^{k-1}\frac{N-X}{N}.\eea

We take the derivative of the CDF in order to find the PDF. After standard integration, we find \bea \widehat{N} \ = \ m_{k-1} \cdot \frac{k+1}{k-1}. \eea

\begin{remark}
For the continuous formula using $m_{k-1}$, we see that there is a difference arising from a factor of $-1$ from the discrete case. The same explanation of the difference applies here too. Also, as we saw in the discrete case, because the continuous and discrete formula are essentially the same formulas, the variance for the formula using $m_{k-1}$ is larger.
\end{remark}

\subsection{Continuous weighted formula}:

As we did in the discrete one dimensional case, we see if constructing a statistic that is a linear combination of $m_k$ and $m_{k-1}$ does a better job estimating $N$. Similar to before, the best value is again when $\alpha = 1$, meaning that the formula using only $m_k$ gives the least variance and there is no benefit to including $m_{k-1}$. The calculation is similar to the discrete weighted problem; see Appendix \ref{sec:appendixDlargesttanks}.

\section{Two Dimensional Discrete Generalizations}

We now generalize the German Tank Problem to two dimensions, after which it will be easy to extend to higher dimensions. We look at the discrete and continuous versions of the square and the circle. We find for each problem which statistic gives the best estimate for $N$, and compare the formulas of the discrete two dimensional square to that of the one dimensional discrete square. Note that if we use all the terms in the calculation, we would not get a clean closed form, as we have many different powers of $m$, making the equation uninvertible. In order to avoid this, we approximate using the main term and get formulas for fixed $k$ and $N$ tending to infinity.

\subsection{Square Problem}

We first consider the case of the square from $(1,1)$ to $(N,N)$ as the natural generalization of the one-dimensional set $\{1, \dots, N\}$. Thus there are $N^2$ pairs, we select $k$ of them without replacement. We call the two components the $X$ and the $Y$ list and use the pairs to find the best estimate for $N$. We investigated two statistics: the largest number from the two lists, and a recursive method where we start with a estimate of $N$ and use the largest $L$ to estimate for $N$ again until the value of our estimate for $N$ stabilizes.

\ \\

\subsubsection{Maximum from Lists:} The motivation of looking at the largest observed component in the two dimensional square comes from the one dimensional German tank problem, where looking at the largest tank gave the most accurate estimation. To calculate the formula, we use the CDF method. The CDF method in the discrete case is slightly different from the one in the continuous case. The statistic we look at is the largest observed component of the $X$ list and $Y$ list, which we denote by $M$. To calculate the probability that the largest observed component is at most $m$, instead of taking the derivative to find the PDF as we did in the continuous case, in the discrete case we calculate \be {\rm Prob}(M \le m) \ - \ {\rm Prob}(M \le m-1)\ee to get the ${\rm Prob}({\rm Max}=m)$. The PDF is computed similarly as in previous arguments, giving \bea {\rm PDF}_{M}(m) & \ = \ & {\rm Prob}(M \le m) \ - \ {\rm Prob}(M \le m-1) \nonumber \\ &=& \frac{\binom{m^2}{k}}{\binom{N^2}{k}} \ - \ \frac{\binom{(m-1)^2}{k}}{\binom{N^2}{k}}. \eea

Thus the expected value of $M$ is
\bea \mathbb{E}[M] & \ = \ & \sum_{m= \lceil \sqrt{k} \rceil}^{N} m \cdot {\rm PDF}_{M}(M = m) \nonumber \\ &=& \sum_{m= \lceil \sqrt{k} \rceil}^{N} \frac{m \binom{m^2}{k} \ - \ m \binom{(m-1)^2}{k}}{\binom{N^2}{k}} \nonumber \\ &=& \sum_{m= \lceil \sqrt{k} \rceil}^{N} \frac{m \binom{m^2}{k} \ - \ (m-1) \binom{(m-1)^2}{k}}{\binom{N^2}{k}} \ - \ \sum_{m= \lceil \sqrt{k} \rceil}^{N} \frac{\binom{(m-1)^2}{k}}{\binom{N^2}{k}} \nonumber \\ && {\rm We \ telescope \ the \ first \ term.} \nonumber \\ &=& \frac{N \binom{N^2}{k} \ - \ (\lceil \sqrt{k} \rceil -1) \binom{(\lceil \sqrt{k} \rceil-1)^2}{k}}{\binom{N^2}{k}} \ - \ \sum_{m= \lceil \sqrt{k} \rceil}^{N} \frac{\binom{(m-1)^2}{k}}{\binom{N^2}{k}}. \eea

We now determine the second term above; it is
\bea \sum_{m= \lceil \sqrt{k} \rceil}^{N} \frac{\binom{(m-1)^2}{k}}{\binom{N^2}{k}} & \ = \ & \frac{1}{\binom{N^2}{k}}\sum_{m= \lceil \sqrt{k} \rceil}^{N-1} \binom{m^2}{k} \nonumber \\ && {\rm Range \ for \ m \ starts \ at \ \lceil \sqrt{k} \rceil \ because \ if \ m \ is \ \lceil \sqrt{k} \rceil-1, \ we \ get \ 0.}  \nonumber \\ &=& \frac{1}{k!\binom{N^2}{k}} \sum_{m= \lceil \sqrt{k} \rceil}^{N-1} m^2(m^2-1)\cdots(m^2-(k-1)). \eea

We use Lemma \ref{2.9boundslemma} to provide upper and lower bounds for the sum
\bea \sum_{m= \lceil \sqrt{k} \rceil}^{N-1} \bigg(m^{2k} - m^{2k-2}\bigg(\frac{k(k-1)}{2}\bigg)\bigg) \le  \sum_{m= \lceil \sqrt{k} \rceil}^{N-1} m^2(m^2-1)\cdots(m^2-(k-1) \le \sum_{m= \lceil \sqrt{k} \rceil}^{N-1} m^{2k}, \eea and we now use the Euler-Maclaurin formula, Lemma \ref{2.8EulerMaclaurin}, to approximate the sums with integrals, and bound the error of the approximation in terms of the derivative of the function at the boundary points. We take $p \ = \ 2$, as for the range we are studying this gives an excellent bound and an expression that is easy to work with.

\begin{remark}
When we calculate the upper and lower bounds, we see that the first order of the upper and lower are the same, which is why we can set these bounds.
\end{remark}

We state a result for a general sum using Euler-Maclaurin, and then will choose appropriate values of the constants. We can just plug in values of $w$ and $y$, which are the powers of $m$, and we also plug in $a$ and $b$ in order to make the calculation easier. \bea \sum_{i=a}^{b} \bigg[m^{w}-c\cdot m^{y}\bigg] & \ = \ & \int_{a}^{b} (x^w-c \cdot x^y) \,dx \ + \ \frac{a^w+b^w-c \cdot a^y-c \cdot b^y}{2} \nonumber \\ && \ \ \ \ \ + \ \frac{B_2}{2}(f'(b)-f'(a)) \ + \ R_2 \nonumber \\ &=& \frac{b^{w+1}-a^{w+1}}{w+1} \ - \ \frac{c \cdot b^{y+1}-c \cdot a^{y+1}}{y+1} \nonumber \\ && \ \ + \ \frac{a^w+b^w-c \cdot a^{y}-c \cdot b^y}{2} \ + \ \frac{1}{12} \bigg(w \cdot b^{w-1}\nonumber \\ && \ \ \ \ \ - \ c \cdot y \cdot b^{y-1} \ - \ w \cdot a^{w-1} + c \cdot y \cdot a^{w-1}\bigg) \ + \ R_2. \eea
We calculate the remainder term for the general equation:
\bea |R_2| & \ \le \ & \frac{2 \zeta(2)}{(2 \pi)^2} \int_{a}^{b} |f^2(x)| \,dx \nonumber \\ & \le & \frac{1}{12} \cdot \bigg[\int_{a}^{b} \bigg(w(w-1)m^{w-2}-c \cdot y \cdot (y-1)m^{y-2}\bigg) \,dm \bigg) \nonumber \\ & \le & \frac{1}{12}\bigg[w \cdot b^{w-1}-c \cdot y \cdot b^{y-1} - w \cdot a^{w-1}+c \cdot y \cdot a^{y-1}\bigg]. \eea
Now that we've calculated the main term and remainder of the general equation, we calculate the lower and upper bounds using the general Euler-Maclaurin formula.

We first do the lower bound: \bea & & \sum_{m=\lceil \sqrt{k} \rceil}^{N-1}\bigg(m^{2k}-m^{2k-2}\bigg(\frac{k(k-1)}{2}\bigg)\bigg) \nonumber \\ &=& \frac{(N-1)^{2k+1}-(\lceil \sqrt{k} \rceil)^{2k+1}}{2k+1} \ - \ \bigg(\frac{k(k-1)}{2}\bigg)\cdot\frac{(N-1)^{2k-1}-(\lceil \sqrt{k} \rceil)^{2k-1}}{2k-1} \nonumber \\ && \ \ \ + \ \frac{\lceil \sqrt{k} \rceil^{2k}+(N-1)^{2k}}{2} \ - \ \frac{k(k-1)}{2} \cdot \frac{\lceil \sqrt{k} \rceil^{2k-2}+(N-1)^{2k-2}}{2} \nonumber \\ && \ \ + \ \frac{1}{12}\bigg[2k(N-1)^{2k-1}-2k \cdot\ (\lceil \sqrt{k} \rceil)^{2k-1}-k(k-1)^2(N-1)^{2k-3} \nonumber \\ && \ \ +k(k-1)^2(\lceil \sqrt{k} \rceil)^{2k-3}\bigg] \ + \ R_2, \eea
where we have the following bound for the remainder term:
\bea |R_2| & \ \le \ & \frac{1}{12}\bigg[2k\cdot(N-1)^{2k-1} - \frac{k(k-1)}{2}(2k-2)(N-1)^{2k-3} \ - \ 2k \cdot (\lceil \sqrt{k} \rceil)^{2k-1} \nonumber \\ && \ \ + \ \frac{k(k-1)}{2}\cdot(2k-2)\cdot (\lceil \sqrt{k} \rceil)^{2k-3} \bigg] \nonumber \\ & \le & \frac{1}{12} \bigg[2k\cdot(N-1)^{2k-1}-2k(\lceil \sqrt{k} \rceil)^{2k-1}-k(k-1)^2(N-1)^{2k-3} \nonumber \\ && \ +k(k-1)^2(\lceil \sqrt{k} \rceil)^{2k-3} \bigg]. \eea

We now calculate the main and remainder term of the upper bound:
\bea \sum_{m=\lceil \sqrt{k} \rceil}^{N-1}m^{2k} & \ = \ & \frac{(N-1)^{2k+1}-(\lceil \sqrt{k} \rceil)^{2k+1}}{2k+1} \ + \ \frac{(\lceil \sqrt{k} \rceil)^{2k}+(N-1)^{2k}}{2} \nonumber \\ && \ \ + \ \frac{1}{12} \bigg(2k \cdot (N-1)^{2k-1}-2k \cdot(\lceil \sqrt{k} \rceil)^{2k-1} \bigg) \ + \ R_2 \nonumber \\ &=& \frac{(N-1)^{2k+1}-(\lceil \sqrt{k} \rceil)^{2k+1}}{2k+1} \ + \ \frac{(\lceil \sqrt{k} \rceil)^{2k}+(N-1)^{2k}}{2} \nonumber \\ && \ \ + \ \frac{k}{6}\bigg[(N-1)^{2k-1}-(\lceil \sqrt{k} \rceil)^{2k-1} \bigg] \ + \ R_2, \eea
with the remainder term bounded by
\bea |R_2| & \ \le \ & \frac{1}{12}\bigg[2k\cdot(N-1)^{2k-1}-2k\cdot(\lceil \sqrt{k} \rceil)^{2k-1} \bigg] \nonumber \\ &\le & \frac{k}{6} \bigg[(N-1)^{2k-1}-(\lceil \sqrt{k} \rceil)^{2k-1} \bigg]. \eea

Now that we have the upper and lower bounds, we use the main term (which as $N\to\infty$ dominates the error terms) to find an equation of $m$ in terms of $N$ and $k$: \bea \sum_{m=\lceil \sqrt{k} \rceil}^{N-1} m^{2k} & \ \approx \ & \frac{(N-1)^{2k+1}-(\lceil \sqrt{k} \rceil)^{2k+1}}{2k+1} \nonumber \\ && {\rm Because \ we \ assumed \ that \ k \ is \ fixed, \ if \ N \ is \ very \ large} \nonumber \\ &&{\rm the \ other \ terms \ are \ negligible.} \nonumber \\ & \approx & \frac{(N-1)^{2k+1}}{2k+1}. \eea

We plug this estimation back into the formula for $\mathbb{E}[M]$: \bea \mathbb{E}[M] & \ = \ & \frac{N \binom{N^2}{k} \ - \ (\lceil \sqrt{k} \rceil -1) \binom{(\lceil \sqrt{k} \rceil-1)^2}{k}}{\binom{N^2}{k}} \ - \ \sum_{m= \lceil \sqrt{k} \rceil}^{N} \frac{\binom{(m-1)^2}{k}}{\binom{N^2}{k}} \nonumber \\ && {\rm We \ use \ the \ same \ argument \ as \ above} \nonumber \\ && {\rm to \ say \ that \ the \ term \ with \ k \ is \ negligible.} \nonumber \\ &\approx& N \ - \ \sum_{m= \lceil \sqrt{k} \rceil}^{N} \frac{\binom{(m-1)^2}{k}}{\binom{N^2}{k}} \nonumber \\ && {\rm We \ plug \ in \ the \ estimate \ of \ the \ second \ term} \nonumber \\ &\approx& N \ - \ \frac{\frac{(N-1)^{2k+1}}{2k+1}}{\binom{N^2}{k} \cdot k!} \nonumber \\ &\approx& N\bigg[1 \ - \ \frac{1}{2k+1}\bigg(1-\frac{1}{N}\bigg)^{2k+1}\bigg] \nonumber \\ && {\rm We \ expand \ out \ the \ first \ two \ terms \ to \ estimate \ for \ (1-1/N)^{2k+1}.} \nonumber \\ && {\rm If \ we \ use \ more \ terms, \ the \ accuracy \ slightly \ increases, \ but \ the \ equation} \nonumber \\ && {\rm will \ not \ be \ invertible. \ Thus, \ we \ only \ write \ the \ first \ two \ terms.} \nonumber \\ & \approx & N\bigg[1-\frac{1}{2k+1}\bigg(1-\frac{2k+1}{N}\bigg)\bigg] \nonumber \\ &=& N\bigg[\frac{2k}{2k+1}\bigg] \ + \ 1 \eea We obtain a good estimate of $(1-1/N)^{2k+1}$ by using two terms.  Inverting the equation, we get  \bea \widehat{N} \ = \ \frac{2k+1}{2k}(m-1) \eea Now that we have a estimation formula for $\widehat{N}$, we run some simulations to see how accurate this estimation formula is. See Appendix \ref{appendixEmathematicacode} for the simulation. We see that the two dimensional formula does well as the variance is small. In the next subsubsection, we compare the one dimensional formula to the two dimensional formula and see which one does better.

\subsubsection{Comparing Formulas} We compare the one dimensional formula and the discrete square formula to see which one does better. We have to make sure that we are making correct comparisons (apples to apples), because a pair gives two data points whereas a point gives one. Also, we want to make sure that both formulas estimate for $N$. For the $N$ by $N$ square, we pick $k$ pairs, which gives us $2k$ components of $N^2$ pairs. For the one dimensional case, we pick $2k$ tanks from $N^2$ possible tanks. This will give us a estimate for $N^2$, and we take the square root to find the estimate for $N$. By comparing these two quantities, we make sure that we observe $2k$ data points. We run code (see Appendix \ref{appendixEmathematicacode}) to compare the two formulas.

\begin{figure}[h]
    \centering
    \includegraphics[width=16cm, frame]{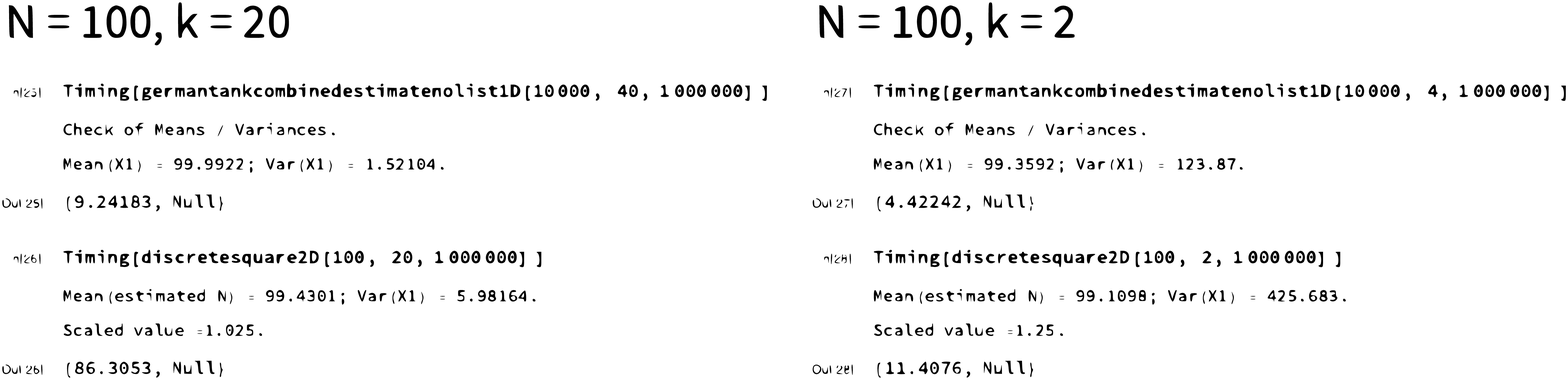}
    \caption{Comparing results from square formula to one dimensional formula}
    \label{fig:One dimensional Problem code}
\end{figure}
\bigskip
\bigskip
\bigskip
\bigskip
\bigskip
\bigskip
\bigskip

\begin{remark}
From the code, we see that the 1-dimensional case does a better job than the two dimensional case, as the one dimensional case has lower variance. The difference is clearly visible when $k$ is a very small value such as 2. We've seen from code that the 1-dimensional formula does better, and we confirm this by theory. We paste some results from the code.
\end{remark}

\subsubsection{Recursive Argument:}
For another approach, we start with an initial estimate for $N$, which we call $N_0$. We construct a formula to recursively generate new estimates of $N$ from previous; thus let $N_1$ be our next guess. We investigate if the values of $N$ converge, and if so if they converges to a more accurate estimate.

We first transform the $N^2$ pairs into a list of numbers from 1 to $N^2$. For each pair ($X$,$Y$), we write the tanks as $(X-1) \ + \ N(Y-1) \ + \ 1$. This way, we can represent all tanks from 1 to $N^2$, as this maps the $N^2$ pairs uniquely to the integers from 1 to $N^2$. Unfortunately, when we use observed tank values to estimate, we do not know the value of $N$ so we cannot immediately use this formula. Instead we replace $N$ with our estimate. With that estimate, we express a new $N$ with the largest observed component. As before, let $M$ denote the maximum of the two lists. Then \bea {\rm Prob}(M = L) \ = \ \frac{\binom{L-1}{k-1}}{\binom{N^2}{k}}. \eea Thus
\bea \mathbb{E}[L] & \ = \ & \sum_{L=k}^{N^2} L \cdot \frac{\binom{L-1}{k-1}}{\binom{N^2}{k}} \nonumber \\ && {\rm We \ calculate \ and \ get} \nonumber \\ &=& k \cdot \frac{N^2+1}{k+1}. \eea Thus
\bea \widehat{N} \ = \ \sqrt{\mathbb{E}[L]\cdot \frac{k+1}{k}-1}. \eea In order to create a iterative process, we plug in $\mathbb{E}[L] \ = \ {\rm Max}X + ({\rm current \ estimation \ for \ N})({\rm Max}Y-1)$. We rewrite our estimation for $\widehat{N}$ \bea \widehat{N} \ = \ \sqrt{\bigg[{\rm Max}X + ({\rm current \ estimation \ for \ N})({\rm Max}Y-1)\bigg] \cdot \frac{k+1}{k}-1}. \eea Now, we've got a recursive function of $N$. We can use the preliminary estimation of $N$ to get another estimate for $N$. Therefore, by starting with a value of $N$, we can continue to produce estimates of $N$. The hope is that by producing more values of $N$, the values converge to the actual number $N$. To see how well this process does, we attached the code for simulation in the Appendix \ref{appendixEmathematicacode}. Using the simulation, we plugged in different values of Max($X$) and Max($Y$). However, though the results converge, they do not do a better job as often the value it converges to is off from the actual $N$.

\subsubsection{\textbf{Continuous Square Problem}} The continuous square problem is slightly different from the discrete square as we can pick points from $(0,0)$ to $(N,N)$. To calculate the PDF, we use the CDF method by calculating the CDF and taking the derivative. Let $m$ denote the largest component observed: \bea CDF_M(m) \ = \ {\rm Prob}(M \le m) \ = \ \left(\frac{m^2}{N^2}\right)^k \ = \ \frac{m^{2k}}{N^{2k}}, \eea and thus \bea {\rm PDF}_M(m) \ = \ {\rm CDF_M}'(m) \ = \ \frac{2k\cdot m^{2k-1}}{N^{2m}}. \eea

Therefore \bea \mathbb{E}[m] & \ = \ & \int_{0}^{N} m \cdot \frac{2k \cdot m^{2k-1}}{N^{2k}}\,dm \nonumber \\ && {\rm We \ calculate \ and \ get:} \nonumber \\ &=& \frac{2k}{2k+1}\cdot N. \eea Solving yields \bea \widehat{N} \ = \ m\cdot \frac{2k+1}{2k}. \eea

\begin{remark}
We see that the scaling factor for the continuous case is the same as the scaling factor for the discrete case. The scaling factor of $(2k+1)/(2k)$ is reasonable, and we see this by comparing this formula to the one dimensional formula. In the one dimensional case, the scaling factor was $(k+1)/(k)$, which is larger than $(2k+1)/(2k)$. However, because in the two dimensional case we are looking at the largest of both components, we have more data points and therefore we will likely get a larger $M$ value. Therefore, in the two dimensional case, we would have to scale by a value smaller than $(k+1)/k$, and scaling by $(2k+1)/(2k)$ makes sense.
\end{remark}

\subsection{Circle Problem}

\subsubsection{\textbf{Discrete Circle Problem}} The goal of the discrete circle problem is to find a formula that estimates the radius. We assume the circle is centered at $(0,0)$ with radius $r$ and we select $k$ different lattice points contained in the circle without replacement. We look at $X^2+Y^2$ as our statistic, because the resulting values are integers, and of course we can then take a square-root at the end.

We let $m_1=X^2+Y^2$. However, some elementary number theory enters in two dimensions and not all values of $m_1$ are attainable. Notice that $X^2$ and $Y^2$ are each 0 or 1 modulo 4, so any attainable $m_1$ is either 0, 1 or 2 modulo 4.

The number of lattice points inside a circle with radius $r$ and center $(0,0)$ is the well studied Gauss Circle problem  \cite{Co}.  Let $P(r)$ be the number of lattice points inside a circle in plane of radius $r$ and center at the origin, i.e., \bea P(r) \ =\ {\rm Number \ of} \ ({(q, n) \in \mathbb{Z}^2 | q^2 + n^2 \le r^2}) . \eea The number of lattice points inside the circle is well estimated by the area of the circle, $\pi r^2$; the challenge is determining the size of the error.

We have \bea P(r) \ = \ \pi r^2 \ + \ E(r). \eea We do not need the best known results, so we write $E(r)$ as $O(r^\delta)$ where $(0\le \delta \le 1)$, using Big-Oh notation (see Definition \ref{2.10bigOnotation}); the current world record has $.5 < \delta < .63$. Note $P(m) \ = \ CDF_{M_{m_1}}(M \le m_1)$, where $M$ is the random variable that is the value of $X^2 + Y^2$. We see that the $PDF_{M}(m)$ is $P(m_1)-P(m_1-1)$. Therefore, if we have $m_1 \equiv 3  \ ({\rm mod \ 4})$ , then the $PDF_{M}(m_1) = 0$, and thus we don't have to worry about this case, though for ease we do include it below.

We calculate the expected value of $M$: \bea \mathbb{E}[M] & \ = \ & \sum_{m_1=0}^{r^2} m_1 \cdot {\rm Prob}(M=m_1) \nonumber \\ &=& [P(1)-P(0)]+[2P(2)-2P(1)]+ \dots + r^2 [P(r^2)-r^2 P(r^2-1)] \nonumber \\ &=& r^2P(r^2)-[P(1)+P(2)+\dots+P(r^2-1) \nonumber \\ &=& r^2 \ - \ \frac{1}{\binom{\pi r^2+O(r^\delta)}{k}} \sum_{m_1=0}^{r^2} \binom{\pi(m-1)+O((m-1)^\delta)}{k}. \eea

We estimate the summation part of $\mathbb{E}[M]$ by using Lemma \ref{2.9boundslemma}: \bea & \ \ & \sum_{m_1=0}^{r^2} \binom{\pi(m_1-1)+O(m_1^\delta)}{k} \nonumber \\ &=& \frac{1}{k!} \sum_{m_1=0}^{r^2} (\pi(m_1-1)+O((m_1-1)^\delta)) \dots (\pi(m_1-1)+O((m_1-1)^\delta)-(k-1)). \eea Note we have
\bea &  & (\pi(m_1-1)+O((m_1-1)^\delta))^k-\frac{k(k-1)}{2}(\pi(m_1-1)+O((m_1-1)^\delta))^{k-1} \nonumber\\ && \ \  \le \ (\pi(m_1-1)+O((m_1-1)^\delta))\dots(\pi(m_1-1)+O((m_1-1)^\delta-(k-1)) \nonumber \\ & \le & (\pi(m_1-1)+O((m_1-1)^\delta))^{k}. \eea

We use Euler-Maclaurin to obtain upper and lower bounds. As the analysis is similar to what we have done earlier, we omit the bounding of the error terms and concentrate on the main term: \bea \sum_{m_1=0}^{r^2}(\pi(m_1-1)+O((m_1-1)^\delta))^k  & \ \approx \ & \int_{0}^{r^2} (\pi(m_1-1)+O((m_1-1)^\delta))^k \,dm \nonumber \\ && {\rm We \ neglect \ the \ Big \ O \ term, \ as \ it \ doesn't} \nonumber\\ && {\rm affect \ the \ final \ calculation \ by \ much} \nonumber \\ & \approx & \pi^k \int_{0}^{r^2} (m_1-1)^k \nonumber \\ && {\rm After \ evaluating \ the \ integrals, \ we \ get} \nonumber \\ & \approx & \pi^k \cdot \frac{(r^2-1)^{k+1}}{k+1}. \eea Now that we've estimated the main term, we calculate $\mathbb{E}[M]$. \bea \mathbb{E}[M] & \ \approx \ & r^2-\frac{1}{\binom{\pi r^2+O(r^\delta)}{k}\cdot k!} \cdot (\pi^k\cdot\frac{(r^2-1)^{k+1}}{k+1}) \nonumber \\ & \approx & r^2 \ - \ \frac{\pi^k\cdot \frac{(r^2-1)^{k+1}}{k+1}}{\pi^k \cdot (r^2)^k} \nonumber \\ & \approx & r^2\bigg[1-\frac{1}{k+1}\cdot (1-\frac{k+1}{r^2})\bigg] \nonumber \\ & \approx & r^2 \cdot \frac{k}{k+1}+1, \eea  and thus \bea \widehat{r} \ = \ \sqrt{\frac{k+1}{k}(m_1-1)}. \eea

\begin{remark} \label{remark6.2}
We analyze the formula from the discrete circle problem. The formula is quite interesting, because we have a square root involved, and unlike other cases, the continuous and discrete setting have different scaling factors. First, we have the square root of $\mathbb{E}[M]$ which is a value similar to $\sqrt{X^2 + Y^2}$, and is similar to $r$. Also, the scaling factor for the discrete case is $\sqrt{(k+1)/k}$, which is similar to $(2k+1)/(2k)$ as we take the square root because taking the square root decreases the value by a little bit. Though the formula doesn't completely align with the continuous circle, this formula makes a lot of sense, and the difference likely results from the different statistics that we looked at for the discrete and continuous circle.
\end{remark}

\subsubsection{\textbf{Continuous Circle Problem}}

The continuous circle problem has similar conditions as the discrete circle problem, but we can select any points contained in the circle; the points don't necessarily have to be lattice points. We approach the continuous circle problem similarly as the continuous square problem. We look at  $m_2=\sqrt{X^2+Y^2}$ because that is the formula for the radius. Let $m_2$ be the largest observed statistic, and $M$ the corresponding random variable. Then \bea {\rm Prob(\sqrt{X^2+Y^2} \le m_2)} \ = \ \frac{(m_2^2 \pi)^k}{(r^2 \pi)^k} \ = \ \frac{m_2^{2k}}{r^{2k}}. \eea Thus \bea {\rm PDF}_M(m_2) \ = \ {\rm CDF}_M'(m_2) \ = \ \frac{2k\cdot m_2^{2k-1}}{r^{2k}}. \eea
The calculation for the expected value is the same as the continuous square, so we omit it, and we find \bea \widehat{r} \ = \ m_2 \cdot \frac{2k+1}{2k}. \eea

\begin{remark}\label{remark6.3}
We compare the continuous circle formula to the discrete circle formula. They don't look similar, as the discrete formula has a square root. Note that $m_1=m_2^2$, by the values of statistics we look at. Also, if we taylor expand $\sqrt{1+1/k}$, we get $\sqrt{1+1/k} \ = \ 1+1/2k+\dots$. In the discrete circle, as the value of $k$ gets very very large, the $\sqrt{1+1/k}$ looks like $1+1/2k$, which is the formula for the continuous circle. Thus, though the formulas look different, if we take the limit as $k$ gets large, we see how similar these are.
\end{remark}

We run some code to check the formula and see how well it does.

\begin{figure}[h]
    \centering
    \includegraphics[width=16cm, frame]{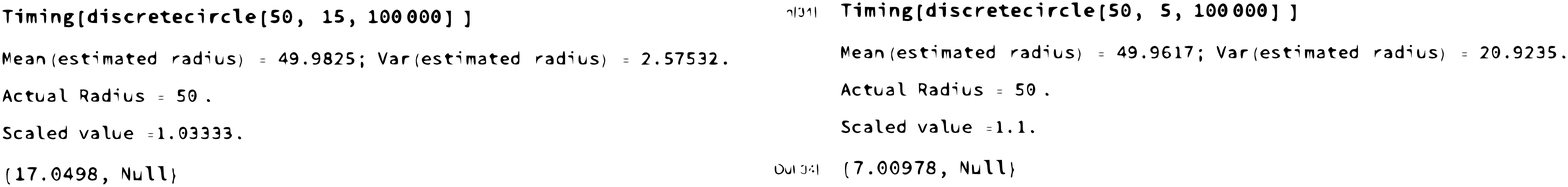}
    \caption{Code for discrete circle.}
    \label{fig:Results of comparison}
\end{figure}
\bigskip
\bigskip
\bigskip

\section{Higher Dimension Version}

\subsection{Generalized Square Problem}

\subsubsection{Discrete Square Problem:} To calculate the formula for the discrete $L$-dimensional square, we use similar strategies as the two dimensional square. We let $M$ be the the largest observed coordinate. Thus
\bea {\rm Prob}(M=m) & \ = \ & F(m)-F(m-1) \nonumber \\ &=& \frac{\binom{m^L}{k}}{\binom{N^L}{k}} \ - \ \frac{\binom{(m-1)^L}{k}}{\binom{N^L}{k}} \nonumber \\ &=& \frac{\binom{m^L}{k}-\binom{(m-1)^L}{k}}{\binom{N^L}{k}}. \eea

The mean is thus
\bea \mathbb{E}[M] & \ = \ & \sum_{n=\lceil \sqrt[L]{k} \rceil}^{N} m \cdot {\rm PDF}_M(m) \nonumber \\ &=& \sum_{n=\lceil \sqrt[L]{k} \rceil}^{N} \frac{m \binom{m^L}{k}-m\binom{(m-1)^L}{k}}{\binom{N^L}{k}} \nonumber \\ &=& \sum_{n=\lceil \sqrt[L]{k} \rceil}^{N} \frac{m \binom{m^L}{k}-(m-1)\binom{(m-1)^L}{k}}{\binom{N^L}{k}} \ - \ \sum_{n=\lceil \sqrt[L]{k} \rceil}^{N} \frac{\binom{(m-1)^L}{k}}{\binom{N^L}{k}} \nonumber \\ &=& \frac{1}{\binom{N^L}{k}}\bigg[N\binom{N^L}{k}-(\lceil \sqrt[L]{k} \rceil)\binom{\lceil \sqrt[L]{k} \rceil)^L}{k}\bigg] \ - \ \sum_{n=\lceil \sqrt[L]{k} \rceil}^{N} \frac{\binom{(m-1)^L}{k}}{\binom{N^L}{k}}. \eea

We use Lemma \ref{2.9boundslemma} to bound the sums: \bea \sum_{n=\lceil \sqrt[L]{k} \rceil}^{N} \binom{(m-1)^L}{k} & \ = \ & \sum_{n=\lceil \sqrt[L]{k} \rceil}^{N-1} \binom{m^L}{k} \nonumber \\ &=& \frac{1}{k!} \sum_{n=\lceil \sqrt[L]{k} \rceil}^{N-1} (m^L)(m^L-1) \cdots(m^L-(k-1)). \eea

We have
\bea m^{Lk}-\frac{k(k-1)}{2}m^{Lk-L} \ \le \ (m^L)(m^L-1) \cdots(m^L-(k-1) \ \le \ m^{Lk}. \eea We apply Euler-Maclaurin to the lower bound: \bea &  & \sum_{n=\lceil \sqrt[L]{k} \rceil}^{N-1}\bigg[m^{Lk}-\frac{k(k-1)}{2}m^{Lk-L}\bigg] \nonumber \\ & =& \frac{(N-1)^{Lk+1}-(\lceil \sqrt[L]{k} \rceil)^{Lk+1}}{Lk+1} \ - \ \frac{k(k-1)}{2}\frac{(N-1)^{Lk-L+1}-(\lceil \sqrt[L]{k} \rceil)^{Lk-L+1}}{Lk-L+1} \nonumber \\ && + \ \frac{(N-1)^{Lk}+(\lceil \sqrt[L]{k} \rceil)^{Lk}}{2} \ - \ \frac{k(k-1)}{2}\frac{(N-1)^{Lk-L}+(\lceil \sqrt[L]{k} \rceil)^{Lk-L}}{2} \nonumber \\ && \ + \ \frac{1}{12}\bigg[Lk\cdot(N-1)^{Lk-1}-\frac{k(k-1)}{2}(Lk-L)(N-1)^{Lk-L-1} \nonumber \\ && \ - \ Lk \cdot (\lceil \sqrt[L]{k} \rceil)^{Lk-1} \ + \ \frac{k(k-1)}{2}\cdot(Lk-L) \cdot (\lceil \sqrt[L]{k} \rceil)^{Lk-1} \bigg] \ + \ R_2. \eea

We bound the remainder term: \bea |R_2| & \ \le \ & \frac{1}{12} \bigg[Lk \cdot (N-1)^{Lk-1} \ - \ \frac{k(k-1)}{2} \cdot (Lk-L)\cdot (N-1)^{Lk-L-1}\nonumber \\ && \ - \ (Lk)(\lceil \sqrt[L]{k} \rceil)^{Lk-1}+\frac{k(k-1)}{2}(Lk-L)(\lceil \sqrt[L]{k} \rceil)^{Lk-L-1}\bigg]. \eea

Similarly we calculate the main term for the upper bound: \bea && \sum_{n=\lceil \sqrt[L]{k} \rceil}^{N-1} m^{Lk} \nonumber \\ &=& \frac{(N-1)^{Lk+1}-(\lceil \sqrt[L]{k} \rceil)^{Lk+1}}{Lk+1} \ + \ \frac{(N-1)^{Lk}+(\lceil \sqrt[L]{k} \rceil)^{Lk}}{2} \nonumber \\ && + \ \frac{1}{12}\bigg[Lk \cdot (N-1)^{Lk-1}-Lk\cdot(\lceil \sqrt[L]{k} \rceil)^{Lk-1} \bigg] \ + \ R_2. \eea
We bound the remainder term by
\bea |R_2| \ \le \ \frac{1}{12}\bigg[Lk\cdot(N-1)^{Lk-1}-Lk\cdot(\lceil \sqrt[L]{k} \rceil)^{Lk-1} \bigg]. \eea

Now that we have fairly tight upper and lower bounds, we use the main term to find an equation of $m$ in terms of $N$ and $k$. As we are only interested in the case of $k$ fixed and $N$ tending to infinity, we can just keep the main term in the analysis below. We find \bea \sum_{m=\lceil \sqrt[L]{k} \rceil}^{N-1} m^{Lk} & \ \approx \ & \frac{(N-1)^{Lk+1}-(\lceil \sqrt[L]{k} \rceil)^{Lk+1}}{Lk+1} \nonumber \\ && {\rm Because \ we \ assumed \ that \ k \ is \ fixed, \ if \ N \ is \ very \ large} \nonumber \\ &&{\rm the \ term \ with \ k \ is \ negligible.} \nonumber \\ & \approx & \frac{(N-1)^{Lk+1}}{Lk+1} \eea

We plug this estimation back into the formula for $\mathbb{E}[M]$: \bea \mathbb{E}[M] & \ = \ & \frac{N \binom{N^L}{k} \ - \ (\lceil \sqrt[L]{k} \rceil -1) \binom{(\lceil \sqrt[L]{k} \rceil-1)^L}{k}}{\binom{N^L}{k}} \ - \ \sum_{m= \lceil \sqrt[L]{k} \rceil}^{N} \frac{\binom{(m-1)^L}{k}}{\binom{N^L}{k}} \nonumber \\ && {\rm We \ use \ the \ same \ argument \ as \ above} \nonumber \\ && {\rm to \ say \ that \ the \ term \ with \ k \ is \ negligible.} \nonumber \\ &\approx& N \ - \ \sum_{m= \lceil \sqrt[L]{k} \rceil}^{N} \frac{\binom{(m-1)^L}{k}}{\binom{N^L}{k}} \nonumber \\ && {\rm We \ plug \ in \ the \ estimate \ of \ the \ second \ term} \nonumber \\ &\approx& N \ - \ \frac{\frac{(N-1)^{Lk+1}}{Lk+1}}{\binom{N^L}{k} \cdot k!} \nonumber \\ &\approx& N\bigg[1 \ - \ \frac{1}{Lk+1}\bigg(1-\frac{1}{N}\bigg)^{Lk+1}\bigg] \nonumber \\ && {\rm We \ expand \ out \ the \ first \ two \ terms \ to \ estimate \ for \ (1-1/N)^{Lk+1}.} \nonumber \\ && {\rm If \ we \ use \ more \ terms, \ the \ accuracy \ slightly \ increases, \ but \ the \ equation} \nonumber \\ && {\rm will \ not \ be \ invertible. \ Thus, \ we \ only \ write \ the \ first \ two \ terms.} \nonumber \\ & \approx & N\bigg[1-\frac{1}{Lk+1}\bigg(1-\frac{Lk+1}{N}\bigg)\bigg] \nonumber \\ &=& N\bigg[\frac{Lk}{Lk+1}\bigg] \ + \ 1 \eea

Inverting the equation, we get \bea \widehat{N} \ = \ \frac{Lk+1}{Lk}(m-1). \eea

\subsubsection{Continuous Square Problem:} The continuous problem is easily generalized to $L$ dimensions. We select $k$ tuples of length $L$ and let $M$ be the largest observed component: \bea {\rm Prob}(M \le m) \ = \ \left(\frac{m^L}{N^L}\right)^k \ = \ \frac{m^{Lk}}{N^{Lk}}. \eea \bea {\rm PDF}_M(m) \ = \ {\rm CDF}_M'(m) \ = \ \frac{Lk \cdot m^{Lk-1}}{N^{Lk}}. \eea

Thus \bea \mathbb{E}[M] \ = \ \int_{0}^{N} m \cdot \frac{Lk \cdot m^{Lk-1}}{N^{Lk}} \,dm. \eea We solve the integral and we get: \bea \widehat{N} \ = \ m\cdot \frac{Lk+1}{Lk}. \eea

\begin{remark}
We see that the scaling factors for the discrete $L$ dimensional square continuous $L$ dimensional square are the same. This scaling factor makes sense because if we have $L$ dimensions, we have many more components to choose from. Thus by scaling by $(Lk+1)/(Lk)$, which is a value smaller than $(2k+1)/(2k)$, and is close to 1, we get a good estimate for $N$.
\end{remark}

\subsection{Generalized Circle Problem}

\subsubsection{Discrete Circle Problem}: We look at the generalized $L$ dimensional circle problem. We study a statistic similar to the two dimensional circle: \bea X_1^2+X_2^2+ \cdots + X_L^2 \ = \ m_1 \eea We see that all values of $m_1$ are integers. Using this statistic, we end up estimating for $r^2$, so after getting the formula, we have to take the square root of $m_1$ to get a estimation for $r$. Let $P(r)$ be the number of lattice points inside an $L$-ball, a $L$-dimensional sphere with radius $r$. We use the volume of the $L$-ball to find a approximate value for the number of lattice points inside the $L$-ball. The formula for the volume of a $L$- ball is \cite{Gi} \bea V(n) \ = \ \frac{\pi^\frac{L}{2}}{\Gamma(\frac{L}{2}+1)}r^L. \eea We use this formula to find $P(r)$. We denote the bounds with big O notation. \bea P(r) \ = \ \frac{\pi^\frac{L}{2}}{\Gamma(\frac{L}{2}+1)}r^L \ + \ O(r^\delta). \eea We now calculate the density using the CDF method. \bea {\rm PDF}_{M_{m_1}} \ = \ P(m_1)-P(m_1-1). \eea We calculate the expected value \bea \mathbb{E}[M] & \ = \ & \sum_{m_1=0}^{r^2} m_1 \cdot {\rm Prob}(M=m_1) \nonumber \\ &=& \frac{\sum_{m_1=0}^{r^2} m_1 \cdot \binom{\frac{\pi^\frac{L}{2}}{\Gamma(\frac{L}{2}+1)}m_1 \ + \ O(m_1^\delta)}{k} -\binom{\frac{\pi^\frac{L}{2}}{\Gamma(\frac{L}{2}+1)}(m_1-1) \ + \ O((m_1-1)^\delta)}{k}}{{\binom{\frac{\pi^\frac{L}{2}}{\Gamma(\frac{L}{2}+1)}r^L \ + \ O(r^\delta)}{k}}} \nonumber \\ &=& \frac{1}{{\binom{\frac{\pi^\frac{L}{2}}{\Gamma(\frac{L}{2}+1)}r^L \ + \ O(r^\delta)}{k}}} \sum_{m_1=0}^{r^2} \bigg[m_1 \binom{\frac{\pi^\frac{L}{2}}{\Gamma(\frac{L}{2}+1)}m_1 \ + \ O(m_1^\delta)}{k} \nonumber \\ && \ \ - \ (m_1-1) \binom{\frac{\pi^\frac{L}{2}}{\Gamma(\frac{L}{2}+1)}(m_1-1) \ + \ O((m_1-1)^\delta)}{k}\bigg] \nonumber \\ && \ - \ \frac{1}{{\binom{\frac{\pi^\frac{L}{2}}{\Gamma(\frac{L}{2}+1)}r^L \ + \ O(r^\delta)}{k}}} \sum_{m_1=0}^{r^2} \binom{\frac{\pi^\frac{L}{2}}{\Gamma(\frac{L}{2}+1)}(m_1-1) \ + \ O((m_1-1)^\delta)}{k} \nonumber \\ && {\rm We \ telescope \ and \ get} \nonumber \\ &=& r^2 \ - \ \frac{1}{{\binom{\frac{\pi^\frac{L}{2}}{\Gamma(\frac{L}{2}+1)}r^L \ + \ O(r^\delta)}{k}}} \sum_{m_1=0}^{r^2} \binom{\frac{\pi^\frac{L}{2}}{\Gamma(\frac{L}{2}+1)}(m_1-1) \ + \ O((m_1-1)^\delta)}{k}. \eea

We calculate the summation part. \bea & \ \ & \sum_{m_1=0}^{r^2} \binom{\frac{\pi^\frac{L}{2}}{\Gamma(\frac{L}{2}+1)}(m_1-1) \ + \ O((m_1-1)^\delta)}{k} \nonumber \\ &=& \frac{1}{k!} \sum_{m_1=0}^{r^2} \bigg(\frac{\pi^\frac{L}{2}}{\Gamma(\frac{L}{2}+1)}(m_1-1) \ + \ O((m_1-1)^\delta)\bigg)\dots \nonumber \\ && \ \bigg(\frac{\pi^\frac{L}{2}}{\Gamma(\frac{L}{2}+1)}(m_1-1) \ + \ O((m_1-1)^\delta)-(k-1)\bigg). \eea

We use Lemma \ref{2.9boundslemma} to set bounds. \bea & \ \ & \sum_{m_1=0}^{r^2} \binom{\frac{\pi^\frac{L}{2}}{\Gamma(\frac{L}{2}+1)}(m_1-1) \ + \ O((m_1-1)^\delta)}{k}^k \nonumber \\ && \ -\frac{k(k-1)}{2}\sum_{m_1=0}^{r^2} \binom{\frac{\pi^\frac{L}{2}}{\Gamma(\frac{L}{2}+1)}(m_1-1) \ + \ O((m_1-1)^\delta)}{k}^{k-1} \nonumber \\ & \le & \frac{1}{k!} \sum_{m_1=0}^{r^2} \bigg(\frac{\pi^\frac{L}{2}}{\Gamma(\frac{L}{2}+1)}(m_1-1) \ + \ O((m_1-1)^\delta)\bigg) \nonumber \\ && \ \dots \bigg(\frac{\pi^\frac{L}{2}}{\Gamma(\frac{L}{2}+1)}(m_1-1) \ + \ O((m_1-1)^\delta)-(k-1)\bigg) \nonumber \\ & \le & \sum_{m_1=0}^{r^2} \binom{\frac{\pi^\frac{L}{2}}{\Gamma(\frac{L}{2}+1)}(m_1-1) \ + \ O((m_1-1)^\delta)}{k}^k. \eea

We apply Euler-Maclaurin to both bounds. We omit the calculation here, and just show the approximation using the main term. \bea \sum_{m_1=0}^{r^2} \bigg(\frac{\pi^\frac{L}{2}}{\Gamma(\frac{L}{2}+1)}(m_1-1) \ + \ O((m_1-1)^\delta)\bigg)^k & \ \approx \ & \int_{0}^{r^2} \bigg(\frac{\pi^\frac{L}{2}}{\Gamma(\frac{L}{2}+1)}(m_1-1) \ + \ O((m_1-1)^\delta)\bigg)^k \,dm_1 \nonumber \\ && {\rm We \ neglect \ the \ Big \ O \ term} \nonumber \\ & \approx & \int_{0}^{r^2} \bigg(\frac{\pi^\frac{L}{2}}{\Gamma(\frac{L}{2}+1)}(m_1-1)\bigg)^k \,dm_1 \nonumber \\ & \approx \ & \bigg(\frac{\pi^\frac{L}{2}}{\Gamma(\frac{L}{2}+1)}\bigg)^k\cdot \frac{(r^2-1)^{k+1}-(-1)^{k+1}}{k+1} \nonumber \\ & \approx & \bigg(\frac{\pi^\frac{L}{2}}{\Gamma(\frac{L}{2}+1)}\bigg)^k\cdot \frac{(r^2-1)^{k+1}}{k+1}. \eea

We calculate the expected value. \bea \mathbb{E}[M] & \ = \ & r^2 \ - \ \frac{1}{k! \binom{\frac{\pi^\frac{L}{2}}{\Gamma(\frac{L}{2}+1)}r^L \ + \ O(r^\delta)}{k}} \cdot \bigg(\frac{\pi^\frac{L}{2}}{\Gamma(\frac{L}{2}+1)}\bigg)^k\cdot \frac{(r^2-1)^{k+1}}{k+1} \nonumber \\ & \approx & r^2 \ - \ \frac{1}{\bigg(\frac{\pi^\frac{L}{2}}{\Gamma(\frac{L}{2}+1)}\bigg)^k \cdot r^{2k}} \cdot \bigg(\frac{\pi^\frac{L}{2}}{\Gamma(\frac{L}{2}+1)}\bigg)^k \cdot \frac{(r^2-1)^{k+1}}{k+1} \nonumber \\ & \approx & r^2 \ - \ \frac{1}{k+1} \cdot \frac{(r^2-1)^{k+1}}{r^{2k}} \nonumber \\ & \approx & r^2 \bigg[1-\frac{1}{k+1}\bigg(1-\frac{1}{r^2}\bigg)^{k+1}\bigg] \nonumber \\ && {\rm We \ use \ the \ first \ two \ terms \ of \ the \ binomial \ expansion}  \nonumber \\ & \approx & r^2 \bigg[1-\frac{1}{k+1}\bigg(1-\frac{k+1}{r^2}\bigg)\bigg] \nonumber \\ & \approx & r^2 \cdot \frac{k}{k+1}+1. \eea We invert the relationship to get \bea \widehat{r} \ = \ \sqrt{(m_1-1) \cdot \frac{k+1}{k}}. \eea

\begin{remark}
Notice that the formula for discrete $L$-dimensional circle problem isn't dependent on $N$.
\end{remark}


\subsubsection{\textbf{Continuous $L$-dimensional circle problem}} We select $k$ tuples of length $L$ where each component is contained in the $L$-dimensional circle. Let $m$ be the largest observed component. We look at $m_2=\sqrt{X_1^2 \ + \ X_2^2 \ + \cdots + X_L^2}$ as out statistic. We use the volume of the $L^{th}$ dimensional circle to calculate the PDF. Recall the formula of the $L$-ball. We see that \bea V(n) \ = \ \frac{\pi^\frac{L}{2}}{\Gamma(\frac{L}{2}+1)}r^L. \eea

We calculate the CDF: \bea {\rm Prob}(\sqrt{X_1^2 \ + \ X_2^2 \ + \cdots + X_L^2} \le m_2) \ = \ \frac{(\frac{\pi^\frac{L}{2}}{\Gamma(\frac{L}{2}+1)}m_2^L)^k}{(\frac{\pi^\frac{L}{2}}{\Gamma(\frac{L}{2}+1)}r^L)^k} \ = \ \frac{m_2^{Lk}}{r^{Lk}}. \eea
The rest of the calculation is identical as the $L$-dimensional square problem. Thus, we omit the calculation. After calculating, we get: \bea \hat{r} \ = \ m_2 \cdot \frac{Lk+1}{Lk}. \eea

\begin{remark}
Notice that the L dependence in the formula is very small as the formula is $1+1/Lk$ and if L and k are reasonably large, then the scaling factor is very very close to 1.
\end{remark}

\section{Acknowledgements}
I, Anthony Lee, was the student and this research was mentored by Professor Steven J Miller. I truly thank Professor Miller for this research opportunity, the research topic, and his continued hard work in making the research process very intriguing and smooth. I also thank my parents and teachers who have been very supportive and encouraging about my studies.

\appendix
\section{Basic Probability Review}\label{sec:appendixBprob}

\subsection{Variance and Covariance}

\begin{lemma}\label{lem:varformula} The variance of a random variable satisfies
\bea {\rm Var}(X) \ := \ \mathbb{E}[(X-\mathbb{E}[X])^2]  \ = \ \mathbb{E}[X^2] \ - \ \mathbb{E}[X]^2. \eea
\end{lemma}

\begin{proof}
\bea {\rm Var}(X) & \ = \ & \mathbb{E}[(X-\mathbb{E}[X])^2] \nonumber \\ &=& \mathbb{E}[X^2 \ - \ 2X \ \mathbb{E}[X] \ + \ \mathbb{E}[X]^2] \nonumber \\ &=& \mathbb{E}[X^2] \ - \ 2 \ \mathbb{E}[X]\mathbb{E}[X] \ + \ \mathbb{E}[X]^2 \nonumber \\ &=& \mathbb{E}[X^2] \ - \ \mathbb{E}[X]^2. \eea
\end{proof}

\begin{lemma}\label{lem:covariance} The covariance of two random variables satisfies
\bea {\rm Cov}(X,Y) \ = \ \mathbb{E}[X-\mathbb{E}[X]] \cdot \mathbb{E}[Y-\mathbb{E}[Y]] \ = \ \mathbb{E}[XY] \ - \ \mathbb{E}[X]\mathbb{E}[Y]. \eea
\end{lemma}

\begin{proof}
We use the linearity properties of the expected values. We have \bea {\rm Cov}(X,Y) & \ = \ & \mathbb{E}[X-\mathbb{E}[X]] \cdot \mathbb{E}[Y-\mathbb{E}[Y]] \nonumber \\ &=& \mathbb{E}[XY \ - \ X\cdot \mathbb{E}[Y] \ - \ Y \cdot \mathbb{E}[X] \ + \ \mathbb{E}[X]\mathbb{E}[Y]] \nonumber \\ &=& \mathbb{E}[XY] \ - \ \mathbb{E}[X] \cdot \mathbb{E}[Y] \ - \ \mathbb{E}[X] \cdot \mathbb{E}[Y] \ + \ \mathbb{E}[X]\mathbb{E}[Y] \nonumber \\ &=& \mathbb{E}[XY] \ - \ \mathbb{E}[X]\cdot\mathbb{E}[Y].  \eea
\end{proof}

\subsection{Identities from Pascal's Triangle}

\begin{lemma}\label{pascalidentity}(Pascal's Identity) We have \bea \binom{n+1}{r} \ = \ \binom{n}{r} \ + \ \binom{n}{r-1}. \eea  \end{lemma}

\begin{proof}
Pascal's identity can be proved directly by expanding out the factorials, but it can also be proved by story. Say we have $n+1$ people and we want to select $k$ of them. Designate one person as special and the other $n$ as ordinary. There are two cases: we select the special person, in which case we need $k-1$ additional people all drawn from the $n$ ordinary people, or we do not select the special person and now need $k$ ordinary people. As these are mutually exclusive the number of ways, $\binom{n+1}{k}$, is equal to the sum: $\ncr{1}{1}\ncr{n}{k-1} + \ncr{1}{0}\ncr{n}{k}$.
\end{proof}

\begin{lemma}\label{lem:fourhockeystickidentity}(Hockey Stick Identity) We have \bea \sum_{i=r}^{n} \binom{i}{r} \ = \ \binom{n+1}{r+1}. \eea \end{lemma}

We can prove the hockey stick identity using induction and Pascal's identity.

\begin{proof}
Base case: $n = r$ \bea \binom{r}{r} \ = \ \binom{r+1}{r+1} \ = \ 1. \eea

Inductive step: Assume it is true for $n = k$: \bea \sum_{i=r}^{n} \binom{i}{r} \ = \ \binom{n+1}{r+1}. \eea
Using our assumption, we prove the case where $n=k+1$: \bea \sum_{i=r}^{n+1} \binom{i}{r} \ = \ \binom{n+2}{r+1}. \eea

We rewrite the left hand side to facilitate using the inductive hypothesis: \bea \sum_{i=r}^{n+1} \binom{i}{r} & \ = \ & \sum_{i=r}^{n} \binom{i}{r} \ + \ \binom{n+1}{r} \nonumber \\ && {\rm We \ use \ our \ inductive \ hypothesis} \nonumber \\ &=& \binom{n+1}{r+1} \ + \ \binom{n+1}{r} \nonumber \\ &=& \binom{n+2}{r+1} \nonumber \\ && {\rm We \ use \ Pascal's \ identity, \ and \ thus} \nonumber \\ && {\rm we've \ proved \ the \ hockey \ stick \ identity}. \eea
\end{proof}


\section{Portfolio Theory}\label{sec:appendixAportfoliotheory}

Consider the case of two stock $X_1, X_2$ with the same mean return and variances $\sigma_{1}$ and $\sigma_{2}$; the variances are not necessarily the same. For simplicity we assume the two stocks' performances are independent of each other, though in general we need to consider covariances. We construct a weighted portfolio $X_\alpha := \alpha X_1 + (1-\alpha) X_2$, with $\alpha \in [0,1]$. It is easy to see that the expected value of $X_\alpha$ is that of the two stocks; our goal is to find $\alpha$ that minimizes the variance of $X_\alpha$ and thus gives us the most certainty in knowing our future performance. Of course, such a strategy decreases the possibility of getting a much larger than expected return, but it also minimizing the possibility of having a significantly smaller return.

Let’s say we have two options. The first is we are guaranteed  \$500,000. The second is we have a 50\% chance of getting \$1,000,000, and a 50\% chance of getting \$0 dollars. The expected value for both is \$500,000 dollars. Though it may depend on each person’s financial situation, we see that taking the \$500,000 dollars has no risk. For some, this may be life changing (and the marginal utility of the second \$500,000 is almost surely less than the first).

The hypothetical situation above is a simple example of modern portfolio theory, a practical method for selecting investments where we maximize the overall return with an acceptable level of risk. This theory was pioneered by American economist Harry Markowitz; see \cite{In}. A key idea of this theory is diversification. Because most investments are either high risk and high return or low risk and low return, Markowitz argued that perhaps investors could achieve best profit with acceptable risk by choosing an optimal mix of the investments.

The expected return of the portfolio is calculated as a weighted sum of the return of individual assets. Let’s say that we have four assets with different values of expected return. To calculate the risk of the portfolio, an investor needs to know the variances of each assets and also the correlation of any two assets, so six correlation values in total. Because of the asset correlations, the total portfolio risk is lower than what would be calculated by a weighted sum.

We return to the simple case of two stocks both with mean $\mu$, variances $\sigma_i$ (we may assume $0 \le \sigma_2 \le \sigma_1$) and we assume the two stocks are independent. Then if $X_\alpha = \alpha X_1 + (1-\alpha) X_2$ we have \bea {\rm Var}(X_\alpha) \ = \ \alpha^2 \sigma_1^2 + (1-\alpha)^2 \sigma_2^2. \eea To find the minimum variance we check the endpoints ($\alpha = 0$ or 1) and the critical points from the derivative of the variance is zero: that happens when \be 2 \alpha \sigma_1^2 - 2(1-\alpha) \sigma_2^2 \ = \ 0, \ee which gives a critical value of \be \alpha_\ast \ = \ \frac{\sigma_2^2}{\sigma_1^2+\sigma_2^2} \ee We now see which value of alpha gives the minimum variance.
\begin{itemize}
    \item When $\alpha=0$ \ , \ ${\rm Var}(X_\alpha) \ = \ \sigma_2^2$
    \item When $\alpha=1$ \ , \ ${\rm Var}(X_\alpha) \ = \ \sigma_1^2$
    \item When $\alpha=\frac{\sigma_2^2}{\sigma_1^2+\sigma_2^2}$ \ , \bea {\rm Var}(X_\alpha) & \ = \ & \bigg(\frac{\sigma_2^2}{\sigma_1^2+\sigma_2^2}\bigg)^2 \sigma_1^2 \ + \ \bigg(1-\frac{\sigma_2^2}{\sigma_1^2+\sigma_2^2}\bigg)^2 \sigma_2^2 \nonumber \\ &=& \bigg(\frac{\sigma_2^2}{\sigma_1^2+\sigma_2^2}\bigg)^2 \sigma_1^2 \ + \ \bigg(\frac{\sigma_1^2}{\sigma_1^2+\sigma_2^2}\bigg)^2 \sigma_2^2 \nonumber \\ &=& \frac{\sigma_1^2 \sigma_2^2(\sigma_1^2+ \sigma_2^2)}{(\sigma_1^2+\sigma_2^2)^2}  \nonumber \\ &=& \frac{\sigma_1^2 \sigma_2^2}{\sigma_1^2+\sigma_2^2}  \eea
\end{itemize}
From our assumption that $0 < \sigma_2 \le \sigma_1$, we see that $\sigma_2^2  \leq  \sigma_1^2$. Now, we compare $\frac{\sigma_1^2\sigma_2^2}{\sigma_1^2+\sigma_2^2}$ and $\sigma_2^2$.

Straightforward algebra shows \bea \frac{\sigma_1^2\sigma_2^2}{\sigma_1^2+\sigma_2^2} & \  < \ & \sigma_2^2; \eea to see this multiply both sides by $\sigma_1^2+\sigma_2^2$,  and subtract $\sigma_1^2\sigma_2^2$ and obtain $0 < \sigma_2^4$. Thus the variance of the weighted quantity is always less than the variance of the smaller one! If the two variances are equal, the new variance is half of that.

Our German Tank problem is slightly different, as the probability we choose certain tanks is dependent on each other. For example, one condition that we have is that the largest has to be a larger value than the second largest tank, which makes them dependent. We see if we can apply this motivation of using weights to see if we can reduce variance of the outcome even in two dependent random variables.
\section{Proof of Identities}\label{sec:appendixCproofsofidentities}

\begin{proof} \textbf{Identity I.} \bea \sum_{m_k=k}^{N}\binom{m-b}{k-c} & \ = \ & \binom{k-b}{k-b} \ + \ \binom{k-b+1}{k-b} \ + \ \cdots \ + \ \binom{N-b}{k-c} \nonumber\\
&=& \binom{k-c}{k-c} \ + \ \binom{k-c+1}{k-c} \ + \ \cdots \ + \ \binom{N-b}{k-c} \nonumber\\ && \ \ \  - \left[  \binom{k-c}{k-c} \ + \ \binom{k-c+1}{k-c} \ +  \ \cdots \ + \ \binom{k-b-1}{k-c} \right] \nonumber\\
&=& \binom{N-b+1}{k-c+1} \ - \ \binom{k-b}{k-c+1}. \eea \end{proof}

\ \\

\begin{proof} \textbf{Identity II.}

\bea \sum_{m=k-a+1}^{N-a+1} m \frac{\binom{m-1}{k-a}\binom{N-m}{a-1}}{\binom{N}{k}} & \ = \ & \frac{1}{\binom{N}{k}} \sum_{m=k-a+1}^{N-a+1} m \cdot \binom{m-1}{k-a} \binom{N-m}{a-1} \nonumber \\ &=& \frac{1}{\binom{N}{k}} \sum_{m=k-a+1}^{N-a+1} \frac{m!}{(k-a)!(m-k+a-1)!}\cdot \frac{(N-m)!}{(a-1)!(N-m-a+1)!} \nonumber \\ &=& \frac{k-a+1}{\binom{N}{k}}\cdot \sum_{m=k-a+1}^{N-a+1} \binom{m}{k-a+1} \cdot \binom{N-m}{a-1} \nonumber \\ && {\rm We \ use \ Identity \ IV} \nonumber \\ &=& \frac{(N+1)(k-a+1)}{k+1} \eea

\end{proof}

\ \\
\begin{proof} \textbf{Identity III.} \bea \sum_{m=k-a+1}^{N-a+1} m^2 \frac{\binom{m-1}{k-a}\binom{N-m}{a-1}}{\binom{N}{k}} & \ = \ & \frac{1}{\binom{N}{k}}\sum_{m=k-a+1}^{N-a+1}(m+1)m\binom{m-1}{k-a}\binom{N-m}{a-1} \nonumber \\ && \ \ \ \ - \ \frac{1}{\binom{N}{k}}\sum_{m=k-a+1}^{N-a+1}m\binom{m-1}{k-a}\binom{N-m}{a-1} \nonumber \\ &=& \frac{1}{\binom{N}{k}}\sum_{m=k-a+1}^{N-a+1}\frac{(m+1)!}{(k-a)!(m-k+a-1)!}\binom{N-m}{a-1} \nonumber \\ && \ \ \ \ \ - \ \frac{(N+1)(k-a+1)}{k+1} \nonumber \\ &=& \frac{(k-a+1)(k-a+2)(N+2)(N+1)}{(k+2)(k+1)} \nonumber \\ && \ \ \ \ - \ \frac{(N+1)(k-a+1)}{k+1} \eea
\end{proof}

\ \\
\begin{proof} \textbf{Identity IV.}
\bea \binom{a+b+k+1}{a+b+1} \ & = & \ \sum_{i=0}^{k} \binom{a+i}{a} \binom{b+k-i}{b}. \eea We use proof by strong induction to prove the identity. Case where $k=0$: \bea \binom{a+b+1}{a+b+1} \ = \ \binom{a}{a}\binom{b}{b} \ = \ 1 \eea Case where k=1: \bea \binom{a+b+2}{a+b+1} \ = \ \binom{a}{a}\binom{b+1}{b} \ + \ \binom{a+1}{a}\binom{b}{b} \ = \ a+b+2. \eea We continue through all k until $k = k$ and we assume that the case where $k = k$ is true. We use this assumption to prove that the case works for $(k+1)$. We add all the equations when $k=k$ to when $k=1$. We write the sum of equations as the following.
\bea \binom{a}{a}\binom{b+k}{b} \ + \ \binom{a+1}{a}\binom{b+k-1}{b} \ + \ \cdots + \ \binom{a+k}{a}\binom{b}{b} \ = \ \binom{a+b+k+1}{a+b+1}. \eea
\bea \binom{a}{a}\binom{b+k-1}{b} \ + \ \binom{a+1}{a}\binom{b+k-2}{b} \ + \ \cdots + \ \binom{a+k-1}{a}\binom{b}{b} \ = \ \binom{a+b+k}{a+b+1}. \eea
\bea \binom{a}{a}\binom{b+k-2}{b} \ + \ \binom{a+1}{a}\binom{b+k-3}{b} \ + \ \cdots + \ \binom{a+k-2}{a}\binom{b}{b} \ = \ \binom{a+b+k-1}{a+b+1}. \eea
\bea \binom{a}{a}\binom{b+1}{b} \ + \ \binom{a+1}{a}\binom{b}{b} \ = \ \binom{a+b+2}{a+b+1}. \eea
\bea \binom{a}{a}\binom{b}{b} \ = \ \binom{a+b+1}{a+b+1}. \eea We add all the columns. \bea & \ \ & \binom{a}{a}\binom{b+k}{b} \ + \ \binom{a+1}{a}\binom{b+k-1}{b} \nonumber \\ && \ + \ \cdots \ + \ \binom{a+k}{a}\binom{b}{b} \nonumber \\ &=& \ \binom{a+b+k+2}{a+b+2}. \eea The resulting equation is the equation in the $(k+1)$ case. Since we assumed by strong induction that cases when $k=1$ to $k=k$ is all true, we are able to prove that the $(k+1)$ case is true. Therefore, the identity is proved.
\end{proof}

\section{Calculations}\label{sec:appendixDlargesttanks}

\subsection{Formula for $L\textsuperscript{th}$ largest tank} We state the calculation for the estimation formula using the $L\textsuperscript{th}$ largest tank.
\bea {\rm Prob}(M_{k-L+1}=m_{k-L+1}) \ = \ \frac{\binom{m_{k-L+1}-1}{k-L}\binom{N-m_{k-1}}{L-1}}{\binom{N}{k}}. \eea We calculate the expected value of $M_{k-L+1}$. \bea \mathbb{E}[M_{k-L+1}] & \ = \ & \sum_{m_{k-L+1}=k-L+1}^{N-L+1} m_{k-L+1} {\rm Prob}(M_{k-L+1} \ = \ m_{k-L+1}) \nonumber \\ &=& \sum_{m_{k-L+1}=k-L+1}^{N-L+1} m_{k-L+1} \frac{\binom{m_{k-L+1}-1}{k-L} \binom{N-m_{k-L+1}}{L-1}}{\binom{N}{k}} \nonumber \\ &=& \frac{1}{\binom{N}{k}}\sum_{m_{k-L+1}=k-1}^{N-L+1} \frac{k-L+1}{k-L+1} \frac{m_{k-L+1}!}{(k-L)!(m_{k-L+1}-k+L-1)!}\binom{N-m_{k-1}}{L-1} \nonumber \\ &=& \frac{k-L+1}{\binom{N}{k}}\sum_{m_{k-L+1}=k-L+1}^{N-L+1}  \frac{m_{k-L+1}!}{(k-L+1)!(m_{k-L+1}-k+L-1)!}\binom{N-m_{k-1}}{L-1} \nonumber \\ &=& \frac{k-L+1}{\binom{N}{k}}\sum_{m_{k-L+1}=k-L+1}^{N-L+1} \binom{m_{k-L+1}}{k-L+1}\binom{N-m_{k-L+1}}{L-1} \nonumber \\ && {\rm We \ use \ Identity \ IV} \nonumber \\ &=& \frac{k-L+1}{\binom{N}{k}} \binom{N+1}{k+1} \nonumber \\ &=& (N+1)\frac{k-L+1}{k+1}. \eea Therefore, the formula for $N$ in terms of $m$ and $k$ is \bea N \ = \ m_{k-L+1} \frac{k+1}{k-L+1} \ - \ 1. \eea

\begin{remark}
We see that when $L =1$, the case where we observe the largest tank, we get the original German Tank problem formula. Also, notice that as we observe smaller and smaller tanks, the scaling factor increases.
\end{remark}

\subsection{Variance calculation for $m_{k-1}$ in discrete one dimensional case} We start by calculating the variance of $m_{k-1}$
\bea {\rm Var}(M_{k-1}) & \ = \ & \mathbb{E}[M_{k-1}^2] \ - \ \mathbb{E}[M_{k-1}]^2 \nonumber \\ &=& \sum_{m_{k-1}=k-1}^{N-1}m_{k-1}^2 {\rm Prob(M_{k-1} = m_{k-1})} \ - \ \bigg[\sum_{m_{k-1}=k-1}^{N-1}m_{k-1} {\rm Prob(M_{k-1} = m_{k-1})}\bigg]^2 \nonumber \\ && {\rm We \ know \ the \ second \ term \ so \ we \ just \ calculate \ the \ first \ term.} \nonumber \\ &=& \sum_{m_{k-1}=k-1}^{N-1}m_{k-1}^2 \frac{\binom{m_{k-1}-1}{k-2}(N-m_{k-1})}{\binom{N}{k}}
\ - \ \frac{(N+1)^2(k-1)^2}{{k+1}^2}. \eea
We take out the first term and calculate it separately. We expand the binomial coefficients, and regroup terms, so that we can use our identities. \bea & \  \ & \sum_{m_{k-1}=k-1}^{N-1}m_{k-1}^2 \frac{\binom{m_{k-1}-1}{k-2}(N-m_{k-1})}{\binom{N}{k}} \nonumber \\ &=& \sum_{m_{k-1}=k-1}^{N-1}(m_{k-1})(m_{k-1}+1) \frac{\binom{m_{k-1}-1}{k-2} \binom{N-m_{k-1}}{1}}{\binom{N}{k}} \nonumber \\ && \ - \sum_{m_{k-1}=k-1}^{N-1}(m_{k-1}) \frac{\binom{m_{k-1}-1}{k-2} \binom{N-m_{k-1}}{1}}{\binom{N}{k}} \nonumber \\ &=& \frac{1}{\binom{N}{k}} \sum_{m_{k-1}=k-1}^{N-1} \frac{k(k-1)}{k(k-1)} \frac{(m_{k-1}+1)!}{(k-2)!(m_{k-1}-k+1)!} \binom{N-m_{k-1}}{1} \  \nonumber \\ && \ \ \ - \ \frac{1}{\binom{N}{k}} \sum_{m_{k-1}=k-1}^{N-1} \frac{k-1}{k-1} \frac{m_{k-1}!}{(k-2)!(m_{k-1}-k+1)!} \binom{N-m_{k-1}}{1} \nonumber \\ &=& \frac{k(k-1)}{\binom{N}{k}} \sum_{m_{k-1}=k-1}^{N-1} \frac{(m_{k-1}+1)!}{k! (m_{k-1}-k+1)!} \binom{N-m_{k-1}}{1} \  \nonumber \\ && \ \ \ - \ \frac{k-1}{\binom{N}{k}} \sum_{m_{k-1}=k-1}^{N-1} \frac{m_{k-1}!}{(k-1)!(m_{k-1}-k+1)!} \binom{N-m_{k-1}}{1} \nonumber \\ &=& \frac{k(k-1)}{\binom{N}{k}} \sum_{m_{k-1}=k-1}^{N-1} \binom{m_{k-1}+1}{k} \binom{N-m_{k-1}}{1} \  \nonumber \\ && \ \ \ - \ \frac{k-1}{\binom{N}{k}} \sum_{m_{k-1}=k-1}^{N-1} \binom{m_{k-1}}{k-1} \binom{N-m_{k-1}}{1} \nonumber \\ && {\rm We \ use \ Identity \ IV \ to \ simplify} \nonumber \\ &=& \frac{k(k-1)}{\binom{N}{k}} \cdot \binom{N+2}{k+2} \ - \ \frac{k-1}{\binom{N}{k}}\binom{N+1}{k+1}\eea We calculate all terms now. \bea {\rm Var}(M_{k-1}) & \ = \ & \frac{(k)(k-1)(N+1)(N+2)}{(k+1)(k+2)} \ - \ \frac{(N+1)(k-1)}{k+1} \ - \ \frac{(N+1)^2(k-1)^2}{(k+1)^2} \nonumber \\ &=& \frac{2(k-1)(N-k)(N+1)}{(k+1)^2(k+2)}. \eea
We scale the variance of $X_{k-1} \ = \ m_{k-1} (k+1)/(k-1):$
\bea {\rm Var} (X_{k-1}) & \ = \ & {\rm Var} (M_{k-1}) \cdot \frac{(k+1)^2}{(k-1)^2} \nonumber \\ &=& \frac{2(k-1)(N-k)(N+1)}{(k+1)^2(k+2)} \cdot \frac{(k+1)^2}{(k-1)^2} \nonumber \\ &=& \frac{2(N-k)(N+1)}{(k+2)(k-1)}. \eea

\subsection{Continuous Weighted Problem}: In this subsection, we include the calculation of the weighted problem in the continuous case. Let $X$ be the value between 0 and $N$. We use the same motivation from portfolio theory and see if we can apply the idea here of seeing if we can decrease variance by optimizing the two weights.

\subsubsection{Formulas using $m_k$ and $m_{k-1}$} Because we are in the continuous case, we use the CDF method where we take the derivative of the CDF. We use integrals instead of sums. \bea {\rm CDF}_M(X) \ = \ {\rm Prob} (M_{k-1} \le  X) \ = \ (\frac{X}{N})^k \eea \bea f_X(x) \ = \ {\rm PDF}_M(X) \ =\ {\rm CDF}_M(X)' \ = \ \frac{k X^{k-1}}{N^k} \eea \bea \mathbb{E}[X] \ = \ \int_{0}^{N} X f(X) \,dX  \ = \ \int_{0}^{N} \frac{k X^k}{N^k} \,dX \ = \ \frac{k N^{k+1}}{(k+1) N^k}. \eea \bea     m_k \ = \ \frac{k}{k+1} N \eea \bea N \ = \ m_k(\frac{k+1}{k}). \eea Through similar calculations, we can get the formula of estimation using $m_{k-1}$. \bea N \ = \ m_{k-1} (\frac{k+1}{k-1}). \eea

\subsubsection{Variance of $m_k$ and $m_{k-1}$}: \bea {\rm Var}(m_k) & \ = \ &  \mathbb{E}[X_k^2] \ - \ \mathbb{E}[X_k]^2 \nonumber \\ &=& \int_{0}^{N} X^2 f(X) \,dX  \ - \ \Bigg[\int_{0}^{N} X f(X) \,dX \Bigg]^2 \nonumber \\ &=& \int_{0}^{N} \frac{k X^{k+1}}{N^k} \,dX \ - \ (\frac{k N}{k+1})^2 \nonumber \\ &=& \frac{k N^2}{(k+1)^2(k+2)}. \eea We scale the variance of $m_k$ to get the variance of $X_k$ by multiplying $(k+1)^2/k^2$. \bea {\rm Var}(X_k) \ = \ \frac{N^2}{(k+2)k}. \eea After going through very similar calculations, we get the variance of $X_{k-1}$. We won't write the full calculation. \bea {\rm Var}(X_{k-1}) \ = \ \frac{2 N^2}{(k+2)(k-1)}. \eea

\subsubsection{Covariance term} We calculate the covariance term using Lemma \ref{2.4covariancelemma}. Again, we use integrals instead of sums, as we are in the continuous setting.  \bea {\rm Cov}(X_k,X_{k-1}) & \ = \ & \mathbb{E}[X_k \cdot X_{k-1}] \ - \ \mathbb{E}[X_k]\cdot \mathbb{E}[X_{k-1}] \nonumber \\ &=& \frac{(k+1)^2}{k(k-1)}\mathbb{E}[m_k \cdot m_{k-1}] \ - \ \frac{(k+1)^2}{k(k-1)} \mathbb{E}[m_k]\cdot \mathbb{E}[m_{k-1}] \nonumber. \eea We calculate term by term. \bea \mathbb{E}[m_k \cdot m_{k-1}] \ = \ \int_{0}^{N}\int_{0}^{m_k} m_km_{k-1} ({\rm Prob \ M_k \ = \ m_k , \ M_{k-1} = m_{k-1}}) \,dm_{k-1}\,dm_k. \eea To calculate this term, we have to use the joint probability density function. After choosing $m_k$ and $m_{k-1}$, the rest of $(k-2)$ tanks have to be smaller than $m_{k-1}$, the probability is ${m_{k-1}^{k-2}}/{N^k}$. Also, note that we multiply the probability by $\binom{k}{2}$ and also by 2 because the tanks aren't ordered. But, we know that $m_k$ must be larger than $m_{k-1}$. Therefore, we multiply by 2. \bea \mathbb{E}[m_k \cdot m_{k-1}] & \ = \ & 2 \binom{k}{2}\int_{0}^{N}\int_{0}^{m_k} m_km_{k-1} \frac{m_{k-1}^{k-2}}{N^k} \,dm_{k-1}\,dm_k \nonumber \\ && \ {\rm After \ calculating \ this \ integral, \ we \ get:} \nonumber \\ &=& \frac{N^2 (k)(k-1)}{k(k+2)}. \eea We scale $\mathbb{E}[m_k \cdot m_{k-1}]$ to get $\mathbb{E}[X_k \cdot X_{k-1}]$ \bea \mathbb{E}[X_k \cdot X_{k-1}] \ = \ \frac{N^2 (k+1)^2}{k(k+2)}. \eea We know that \bea \mathbb{E}[X_k] \cdot \mathbb{E}[X_{k-1}] \ = \ N^2. \eea Therefore, we calculate the full covariance term. \bea \mathbb{E}[X_k \cdot X_{k-1}] & \ = \ & \Bigg[\frac{N^2 (k+1)^2}{k(k+2)}-1\Bigg] \nonumber \\ &=& \frac{N^2}{k(k+2)}. \eea  Now that we have the variances of $X_k$ and $X_{k-1}$ and covariance of $X_k$ and $X_{k-1}$, we plug them into the formula we used for the discrete weighted sum case.

\subsubsection{Finding optimal alpha}: Recall the weighted statistic. We state the weighted statistic. Let $\alpha \in [0,1]$ and define the weighted statistic $X_\alpha$ by \bea
X_\alpha \ := \ \alpha X_k \ + \ (1-\alpha) X_{k-1}. \eea We take the derivative on both sides and find the alpha value that minimizes the total variance of the weighted statistic. We compute the variance of $X_\alpha$: \bea {\rm Var \ X}_\alpha \ = \ \alpha^2 {\rm Var \ X_k} \ + \ (1-\alpha)^2 {\rm Var \ X_{k-1}} \ + \ 2 \ {\rm Cov}(X_k,X_{k-1}). \eea Let $\alpha_{k,k-1}$ denote the optimal alpha value. Taking the derivative with respect to $\alpha_{k,k-1}$ yields \bea {\rm Var}(X)' & \ = \ & 2\alpha_{k,{k-1}}({\rm Var}(X_k)+{\rm Var}(X_{k-1})-2 \  {\rm Cov}(X_k,X_{k-1}) \nonumber \\ && \ \ \ \ \ \ + \ 2 \  ({\rm Cov}(X_k,X_{k-1}) \ - \ {\rm Var}(X_{k-1})). \eea After doing some basic algebra, we see that the optimal value of alpha is equal to 1. \bea \alpha_{k,k-1} \ = \ \frac{{\rm Var}(X_{k-1}) \ - \ {\rm Cov}(X_k,X_{k-1})}{{\rm Var}(X_{k})+{\rm Var}(X_{k-1})-2 \  {\rm Cov}(X_k,X_{k-1})} \ = \ 1. \eea We plug in the formulas that we've calculated. Then, we see that the optimal alpha value is also 1 for the continuous case.
\pagebreak

\section{Mathematica Code}\label{appendixEmathematicacode}
\textbf{One Dimensional German Tank Problem code}
\begin{figure}[h]
    \centering
    \includegraphics[width=15cm, frame]{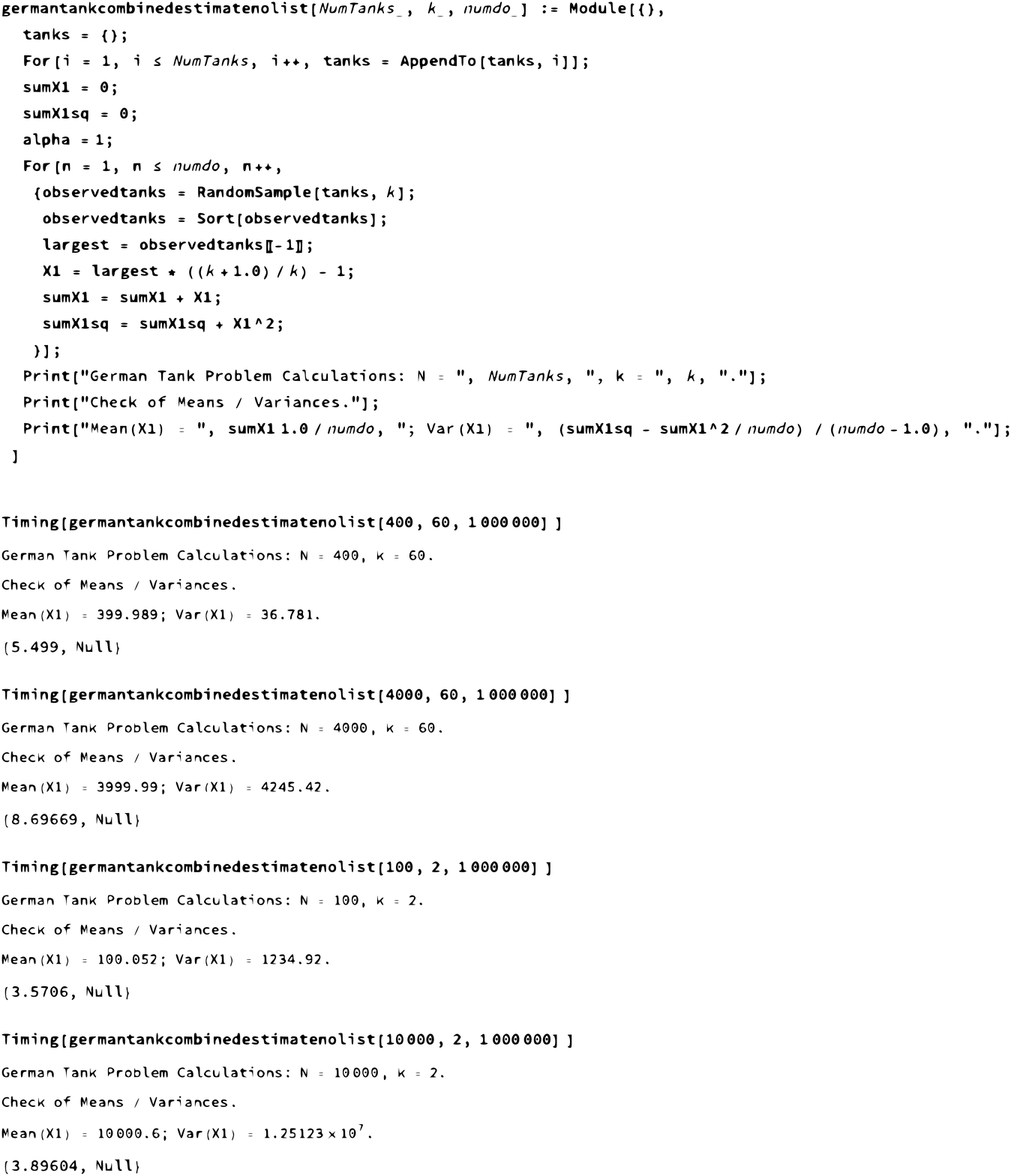}
    \caption{Simulation for one dimensional discrete case.}
    \label{fig:One dimensional Problem code}
\end{figure}
\pagebreak

\textbf{Two Dimensional Circle Code}
\begin{figure}[h]
    \centering
    \includegraphics[width=17cm, frame]{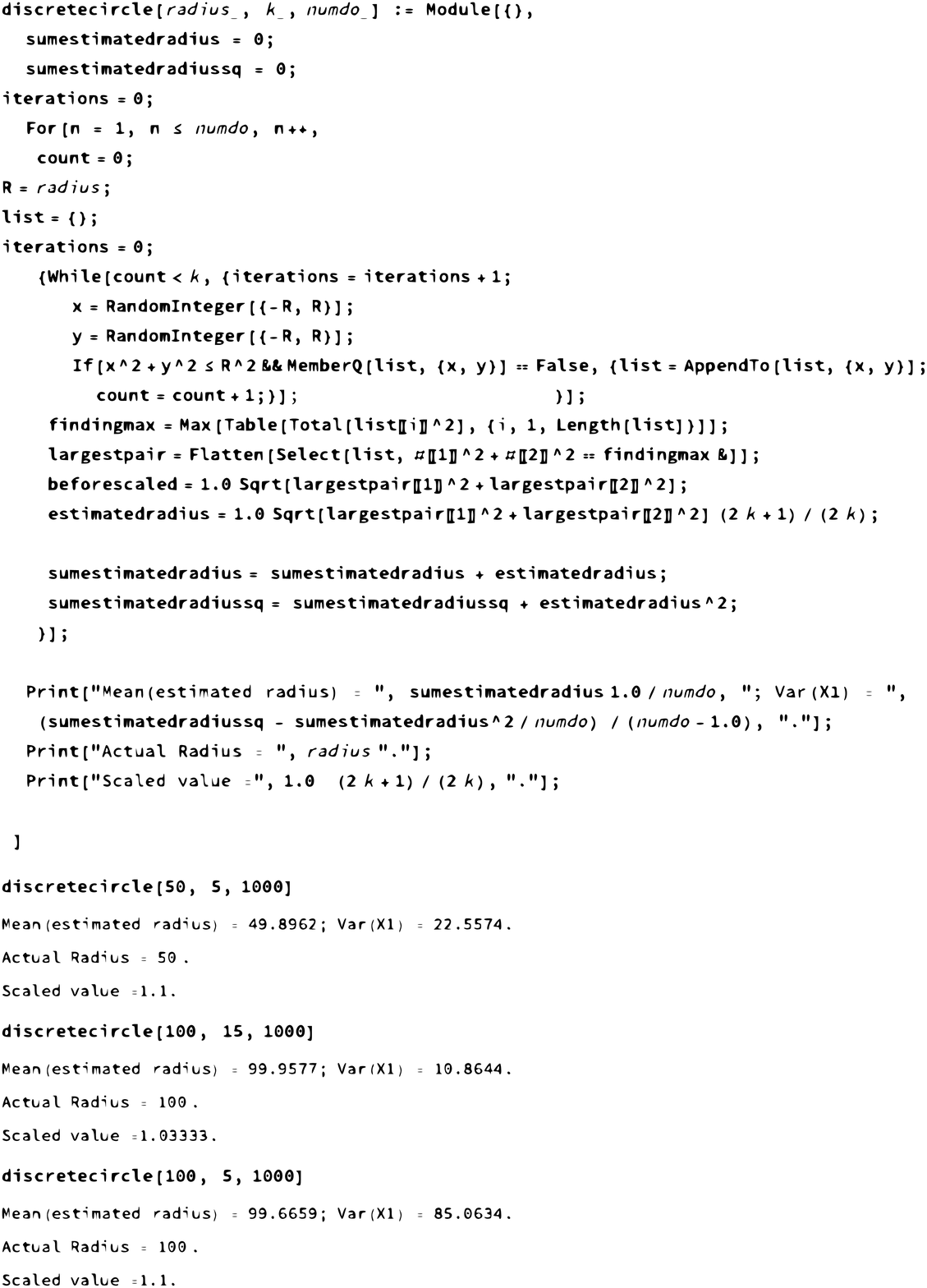}
    \caption{Simulation for two dimensional discrete circle.}
    \label{fig:Two dimensional circle code}
\end{figure}

\textbf{Two Dimensional Discrete Square Simulation}
\begin{figure}[h]
    \centering
    \includegraphics[width=14cm, frame]{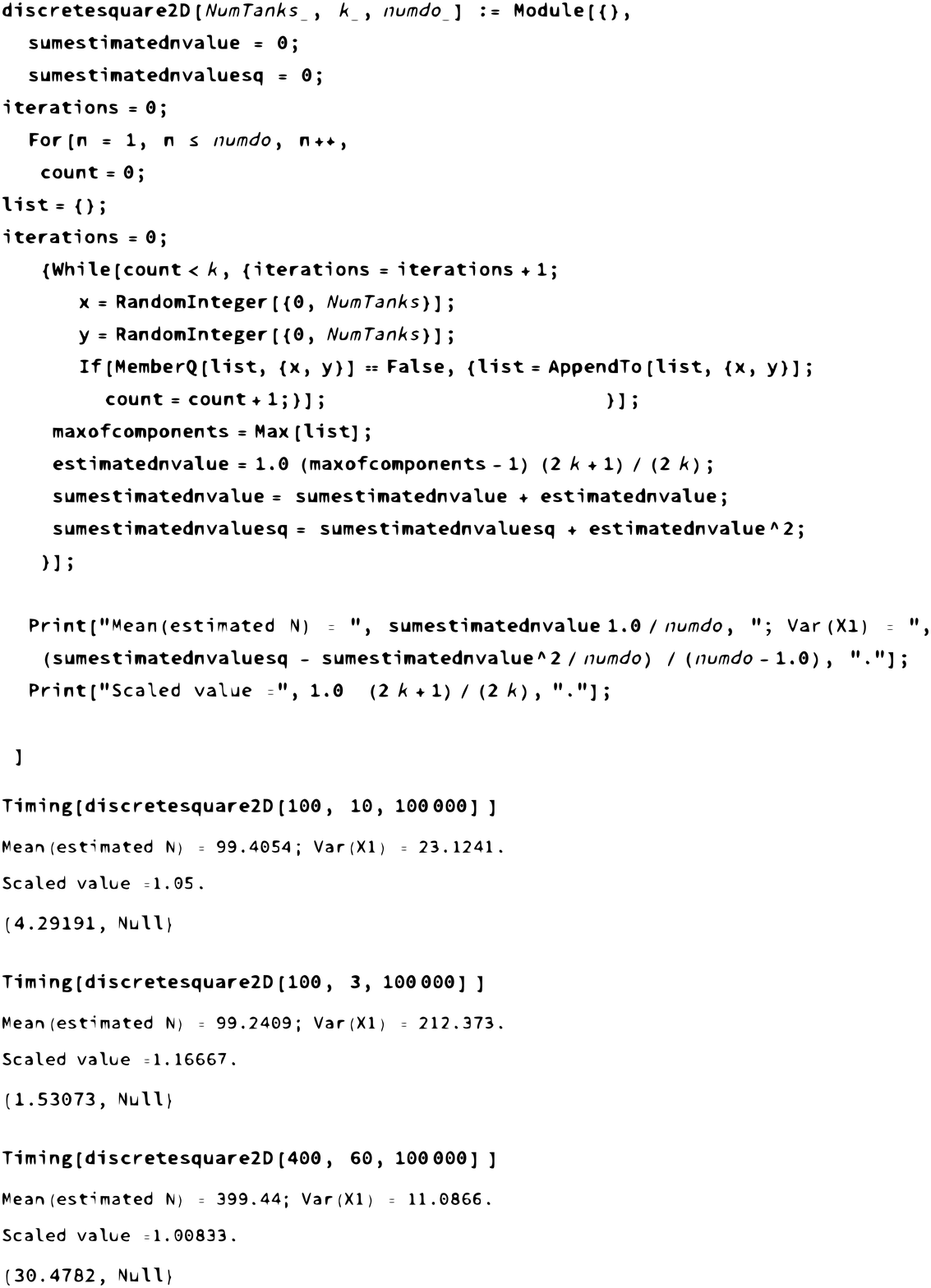}
    \caption{Simulation for Two Dimensional Discrete Square.}
    \label{fig:Two Dimensional Discrete Square}
\end{figure}

\bigskip
\bigskip
\bigskip
\bigskip
\bigskip

\textbf{Recursive Method}
\begin{figure}[h]
    \centering
    \includegraphics[width=11.5cm, frame]{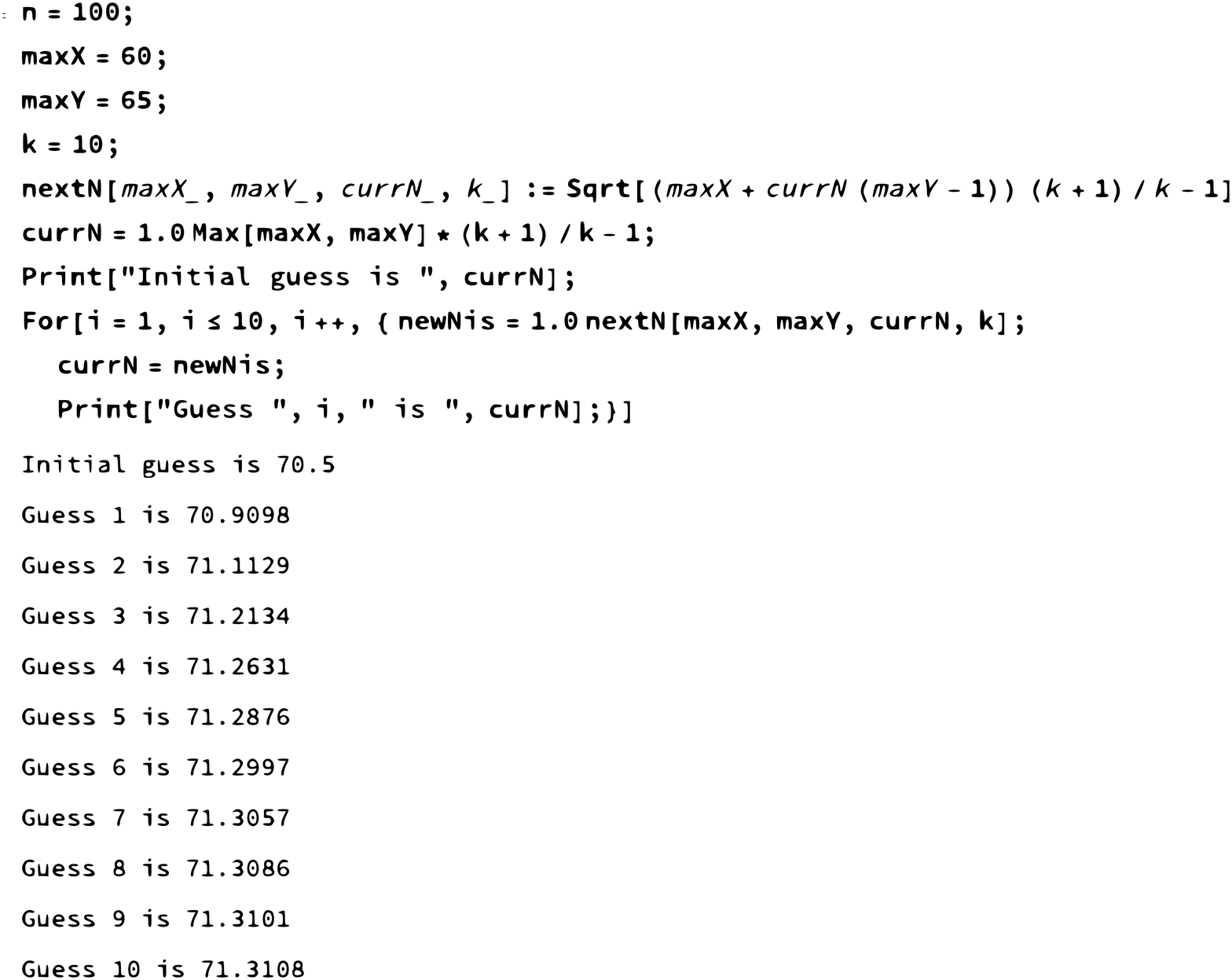}
    \caption{Simulation for recursion method for square.}
    \label{fig:Recursive method}
\end{figure}
\begin{figure}[h]
    \centering
    \includegraphics[width=11.5cm, frame]{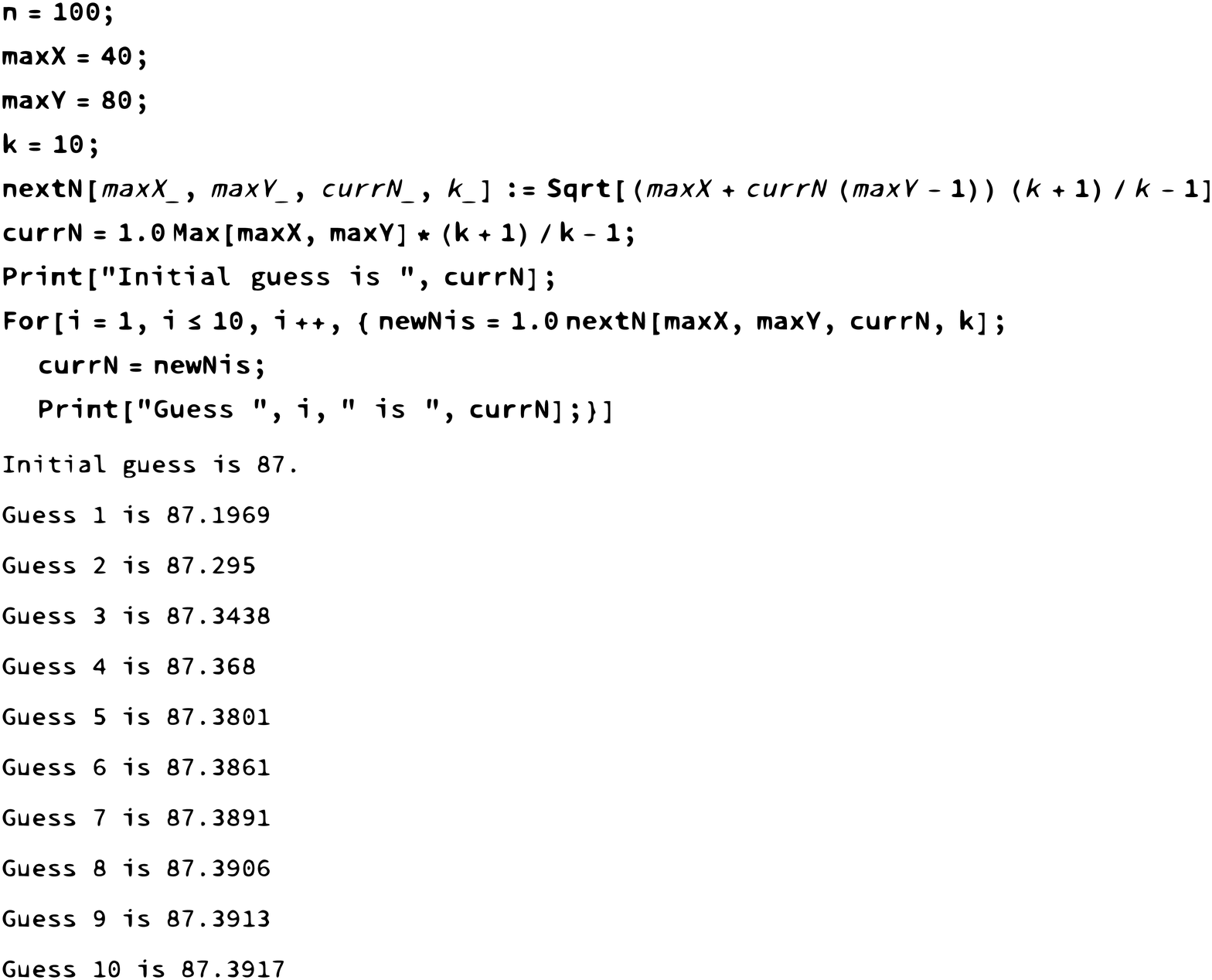}
    \caption{Simulation for two dimensional discrete circle.}
    \label{fig:Two dimensional circle code}
\end{figure}

\ \\

\textbf{Code for Comparing Formulas for one and two dimensions}
\begin{figure}[h]
    \centering
    \includegraphics[width=16.5cm, frame]{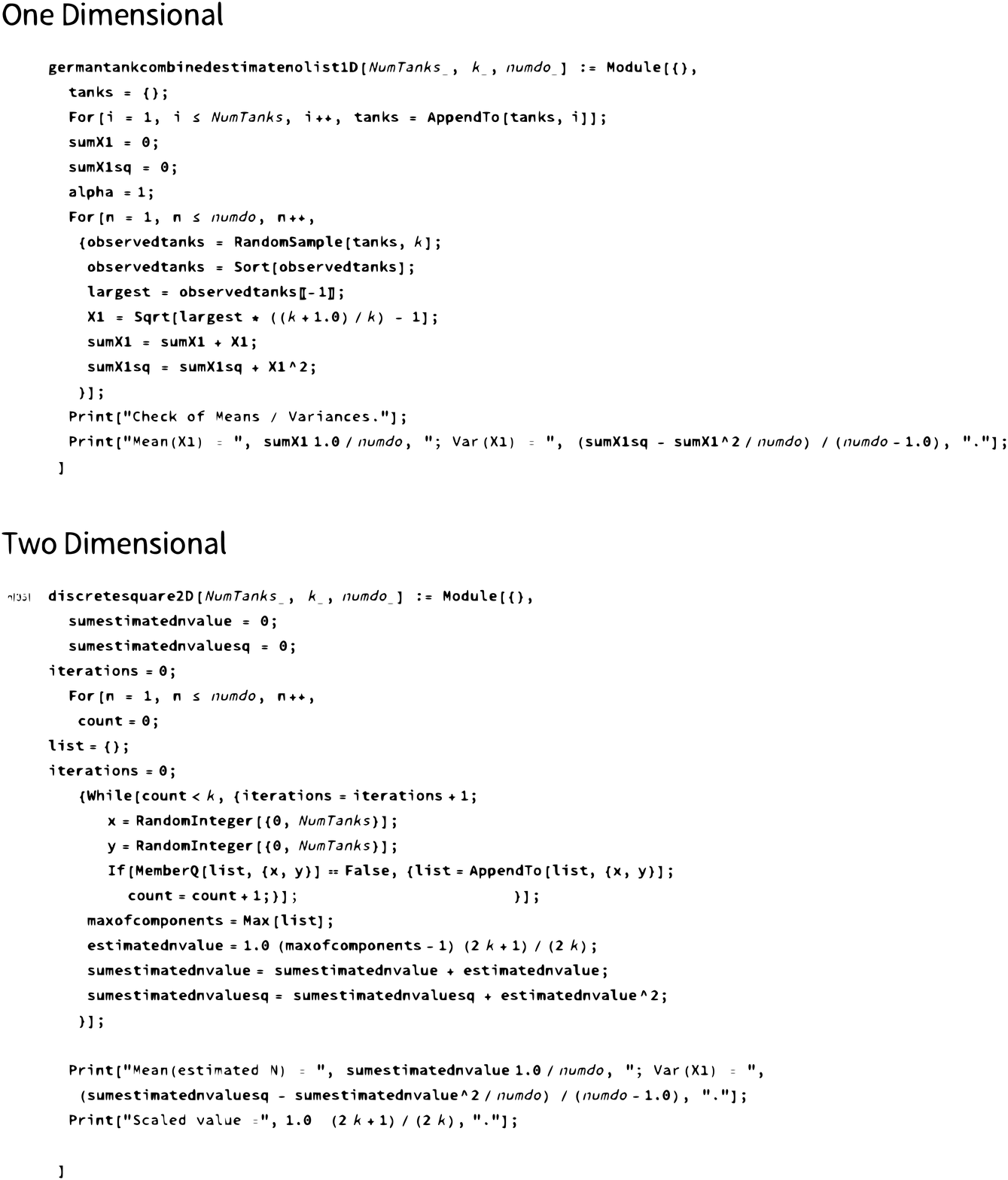}
    \caption{Code for Comparing Formulas for one and two dimensions.}
    \label{fig:Code for Comparing Formulas for one and two dimensions}
\end{figure}

\ \\

\textbf{Results of comparison}
\begin{figure}[h]
    \centering
    \includegraphics[width=16cm, frame]{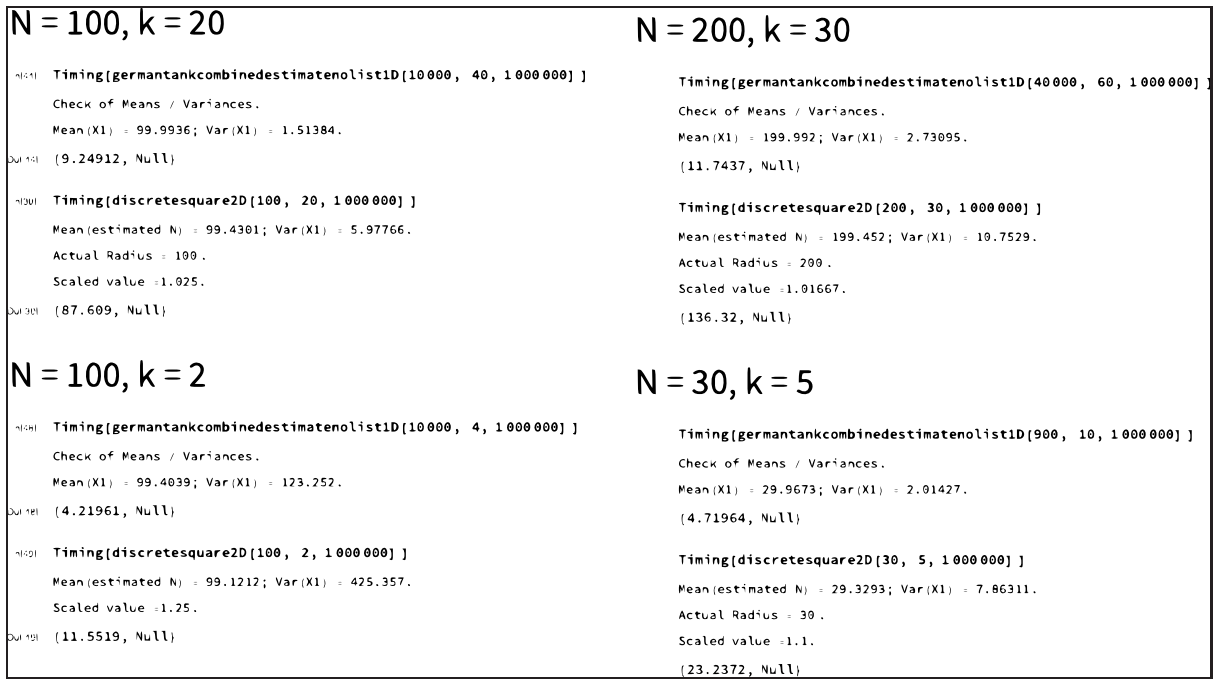}
    \caption{Results of comparison.}
    \label{fig:Results of comparison}
\end{figure}

\pagebreak

\bigskip
\bigskip

\end{document}